\documentclass[a4paper,11pt]{amsart}
\usepackage[T1]{fontenc}
\usepackage[english]{babel}
\usepackage[cp1252]{inputenc}
\usepackage{amsthm}
\usepackage{amsmath}
\usepackage{amsfonts}
\usepackage{amssymb}
\usepackage{hyperref}
\usepackage{color}
\usepackage{caption}
\usepackage{indentfirst}
\usepackage{amssymb}
\usepackage{eufrak}
\usepackage{mathrsfs}
\usepackage{xypic}
\usepackage[pdftex]{graphicx}
\usepackage{tikz}
\usepackage{booktabs}

\theoremstyle{plain}

\newcommand{\R}{\mathbb R}
\newcommand{\C}{\mathbb C}

\newcommand{\Z}{\mathbb Z}

\newcommand{\Log}{\mathrm{Log}\,}

\newcommand{\PP}{\mathbb P}

\newcommand{\di}{\displaystyle}

\newcommand{\Crit}{\operatorname{Crit}}

\newcommand{\Area}{\operatorname{Area}}

\newcommand{\Conv}{\operatorname{Conv}}
 
\newcommand{\MV}{\operatorname{MV}}

\newcommand{\CA}{\mathcal C\mathcal A}
\newcommand{\width}{\operatorname{width}}

\newcommand{\Newt}{\operatorname{Newt}}
\newcommand{\Trop}{\operatorname{Trop}}

\newcommand{\Ram}{\operatorname{Ram}}
\newcommand{\ord}{\operatorname{ord}}
\newcommand{\CP}{\mathbb CP}
\newcommand{\RP}{\mathbb RP}

\newcommand{\Int}{\operatorname{Int}}

\newcommand{\Con}{\operatorname{Con}}

\newcommand{\Rdeg}{\mathbb R\operatorname{deg}}

\newcommand{\Bez}{B_{\mathrm{Bez}}(C,m)}

 \newcommand{\Rea}{\operatorname{Re}}
\newcommand{\Ima}{\operatorname{Im}}

\newcommand{\ConB}{B_{\mathrm{con}}(C,m)}

\newcommand{\LSS}{B_{\rm LSS}}

\newcommand\restr[2]{{% we make the whole thing an ordinary symbol
  \left.\kern-\nulldelimiterspace % automatically resize the bar with \right
  #1 % the function
  \right|_{#2} % this is the delimiter
  }}

\newtheorem{remark}{Remark}[section]
\newtheorem*{mtheorem*}{Main Theorem}
\newtheorem{theorem}{Theorem}[section]
\newtheorem{definition}{Definition}[section]
\newtheorem{proposition}{Proposition}[section]
\newtheorem{corollary}{Corollary}[section]

\begin{document}
\title{Sparse Bounds for Amoeba Contours} 
\author{Mounir Nisse}
\date{}

\email{\href{mailto:mounir.nisse@gmail.com}{mounir.nisse@gmail.com}}

\thanks{Research of M. Nisse is supported in part by Xiamen University Malaysia Research Fund (Grant no. XMUMRF/ 2020-C5/IMAT/0013).}

%%%%%%%%%%%%%%%%%%%%%%%%%%%%%%%%%%%%%%%%%%%%%%%%%%%%%%%%%%%%%%%%%%%%%%%%%%%%%

\subjclass[2010]{14T20, 14M25, 14P25, 14Q20, 32A60, 52B20}
\keywords{Contour of an amoeba, logarithmic Gauss map, real contour degree, sparse elimination, mixed volume, Bernstein's theorem, Pfaffian manifold}

\maketitle

\begin{abstract}
We derive new upper bounds for the real degree of the contour of the amoeba of algebraic curves and hypersurfaces. Our approach refines the Pfaffian method of Lang--Shapiro--Shustin by replacing large total-degree estimates with logarithmic-conormal elimination and sparse mixed-volume techniques based on the transformed Newton polytopes. This yields universal bounds together with significantly sharper sparse estimates for broad classes of Laurent polynomials. Several comparisons and explicit examples illustrate the improvement of the new bounds over the previously known universal estimates.
\end{abstract}
 
%%%%%%%%%%%%%%%%%%%%%%%%%%%%%%%%%%%%%%%%%%%%%%%%%%%%%%%%%%%%%%%%%%%%%%%%%

\section{Introduction}

The geometry of amoebas lies at the intersection of algebraic geometry, complex analysis, convex geometry, and tropical geometry. Since the pioneering work of Gelfand, Kapranov and Zelevinsky \cite{GKZ94}, Bergman
\cite{Bergman71}, Bieri and Groves \cite{BieriGroves84}, and Forsberg,
Passare and Tsikh \cite{ForsbergPassareTsikh00}, the logarithmic image of an algebraic
variety has provided a remarkable link between algebraic geometry,
convex geometry and tropical geometry. Amoebas encode subtle information
on Newton polytopes, logarithmic limit sets, and tropicalizations, while
their complements exhibit rich topological and combinatorial structures.
Subsequent contributions by Mikhalkin \cite{Mikhalkin04}, Passare and
Rullg\aa rd \cite{PassareRullgard04}, Nisse and Sottile
\cite{NisseSottile13}, and the exposition of Maclagan and Sturmfels
\cite{MaclaganSturmfels15} have established amoeba theory as a
fundamental part of modern tropical geometry.

Among the geometric invariants associated with an amoeba, the contour
occupies a particularly important place. It is defined as the set of
critical values of the logarithmic map restricted to the smooth locus of
the hypersurface and coincides with the discriminant of the logarithmic
projection. The contour determines where the logarithmic projection
changes its local covering behaviour and therefore controls many
geometric properties of the amoeba. For smooth hypersurfaces, the
logarithmic critical locus is described by the logarithmic Gauss map,
whose real locus projects onto the contour
\cite{Mikhalkin04,PassareRullgard04}. Consequently, the geometry of the
contour is intimately related to both the logarithmic Gauss map and the
Newton polytope of the defining Laurent polynomial.

A natural quantitative invariant is the real degree of the contour,
namely the maximum number of intersection points between the contour and
a generic affine real line. Despite its simple definition, estimating
this invariant is a difficult problem because it involves the
interaction of logarithmic geometry, elimination theory and real
algebraic geometry. The first general answer was obtained by Lang,
Shapiro and Shustin \cite{LangShapiroShustin21}, who introduced an
elegant Pfaffian approach based on Khovanskii's theory of simple
Pfaffian manifolds \cite{Khovanskii91}. Their work provides explicit
universal bounds depending only on the degree of the defining polynomial
and independent of its coefficients. These estimates constitute the
current benchmark for universal contour-degree bounds.

Although universal, the Lang--Shapiro--Shustin estimates are necessarily
very large. Their construction relies on total-degree estimates and
therefore ignores the sparse structure of Laurent polynomials. In many
situations the Newton polytope contains far fewer lattice points than
the simplex determined by the total degree, suggesting that the true
complexity of the contour should be governed by sparse geometry rather
than by degree alone. Bernstein's theorem \cite{Bernstein75} and the
mixed-volume theory developed in \cite{GKZ94} provide precisely the
appropriate framework for exploiting this additional combinatorial
information.

The purpose of the present paper is to obtain substantially sharper
upper bounds for the real degree of the contour of amoebas of algebraic
curves and hypersurfaces. Our objective is not to replace the Pfaffian
philosophy of Lang, Shapiro and Shustin, but to refine the algebraic
elimination underlying their method. The guiding principle is that the
transformed logarithmic critical equations possess considerably smaller
Newton polytopes than those predicted by total-degree considerations. By
determining these transformed supports explicitly and combining them
with Bernstein's theorem, one obtains sparse estimates that remain
universal while reflecting the actual combinatorial complexity of the
defining Laurent polynomial.

The first main result establishes a new family of universal
contour-degree bounds obtained from logarithmic-conormal elimination
systems. Instead of estimating the transformed equations by their total
degrees, we analyze their exact Newton polytopes and derive
corresponding mixed-volume estimates. This produces universal bounds
that improve the classical Pfaffian estimates for large classes of
sparse hypersurfaces. In particular, the resulting bounds depend
naturally on the transformed supports rather than solely on the degree
of the original polynomial.

A second contribution is a systematic comparison between different
approaches to contour-degree estimates. We compare the universal
Pfaffian bounds of Lang--Shapiro--Shustin, Bézout-type estimates,
conormal elimination bounds and sparse mixed-volume estimates. Explicit
computations for families of plane curves and higher-dimensional
hypersurfaces illustrate that sparse elimination frequently yields a
significant reduction in the predicted contour degree. These examples
demonstrate that the Newton polytope itself carries much finer
information than the total degree regarding the complexity of the
contour.

Another feature of the paper is the explicit analysis of transformed
Newton polytopes arising from logarithmic critical equations. Their
geometry is described chart by chart and compared with the large
total-degree simplices used in previous work. This comparison explains
geometrically why sparse elimination improves universal estimates and
clarifies the relation between logarithmic criticality, Newton polytopes
and mixed volumes.

The results presented here complement the universal theory of Lang,
Shapiro and Shustin rather than replacing it. Their Pfaffian framework
remains the starting point of our analysis, whereas our contribution
consists in refining the algebraic input to that framework by
incorporating sparse elimination and Bernstein theory. In this way we
preserve the universality of the original method while producing
significantly smaller estimates whenever the defining Laurent polynomial
possesses substantial sparsity.

The paper is organized as follows. After recalling the necessary
background on amoebas, logarithmic Gauss maps, Newton polytopes and
Pfaffian manifolds, we derive transformed logarithmic critical systems
and study their Newton polytopes. We then establish new universal and
sparse contour-degree bounds, compare them with the estimates of Lang,
Shapiro and Shustin, and conclude with explicit examples showing the
effectiveness of the proposed method.

{\it Acknowledgements.}  The author would like to express his sincere gratitude to Boris Shapiro for kindly sending his paper with Lionel Lang and Eugeni Shustin \cite{LangShapiroShustin21}. Its results and perspective have been a valuable source of motivation for the present work.

%%%%%%%%%%%%%%%%%%%%%%%%%%%%%%%%%%%%%%%%%%%%%%%%%%%%%%%%%%%%%%%%%%%%%%%%%

\section{Preliminaries}

The theory of amoebas occupies a central position in modern algebraic and tropical geometry. Since the pioneering work of Gelfand, Kapranov and Zelevinsky \cite{GKZ94}, Bergman \cite{Bergman71}, Bieri and Groves \cite{BieriGroves84}, and Forsberg, Passare and Tsikh \cite{ForsbergPassareTsikh00}, the logarithmic image of an algebraic variety has proved to encode remarkable geometric, topological and combinatorial information. Amoebas provide a natural bridge between algebraic varieties and tropical geometry, while their asymptotic behaviour is governed by logarithmic limit sets and Newton polytopes. Further developments by Mikhalkin \cite{Mikhalkin04}, Passare and Rullg{\aa}rd \cite{PassareRullgard04}, Nisse and Sottile \cite{NisseSottile13}, and the monograph of Maclagan and Sturmfels \cite{MaclaganSturmfels15} established the logarithmic map as one of the fundamental tools in tropical geometry.

Throughout this paper we restrict our discussion to algebraic curves and hypersurfaces in the algebraic torus $
(\mathbb C^\ast)^n.
$ Let $
H=\{f=0\}\subset(\mathbb C^\ast)^n
$ be a reduced hypersurface defined by a Laurent polynomial $
f=\sum_{\alpha\in\mathbb Z^n}c_\alpha z^\alpha.
$ The logarithmic map is $
\Log(z_1,\ldots,z_n)=
(\log|z_1|,\ldots,\log|z_n|),
$ and the amoeba of $H$ is $
\mathcal A_H=\Log(H).
$ The contour of the amoeba is the set of critical values of $
\Log|_{H_{\rm sm}},
$ where $H_{\rm sm}$ denotes the smooth locus. Geometrically, the contour records the singular behaviour of the logarithmic projection and separates regions where the logarithmic map has different covering multiplicities.

For a smooth hypersurface the logarithmic Gauss map is
$
\gamma_{\log}(z)=
[z_1f_{z_1}:\cdots:z_nf_{z_n}]
\in\mathbb P^{n-1}.
$
It was shown by Mikhalkin and Passare--Rullg{\aa}rd that a point is logarithmically critical precisely when its logarithmic normal direction is real, namely when $
\gamma_{\log}(z)\in\mathbb RP^{\,n-1}
$ \cite{Mikhalkin04,PassareRullgard04}. Consequently the contour is obtained by projecting the real part of the logarithmic Gauss correspondence. This description reveals the intimate relation between the contour and the projective geometry of logarithmic normal directions.

The invariant studied in this paper is the real degree of the contour. If $
L
$ is an affine real line transverse to the regular part of the contour, one considers the finite set $
L\cap\mathcal C\mathcal A_H.
$ The real contour degree is defined by
$
\operatorname{Rdeg}(\mathcal C\mathcal A_H)=
\sup_L
\#(L\cap\mathcal C\mathcal A_H).
$
This quantity measures the maximal complexity of the contour with respect to affine line sections and may be regarded as a real analogue of the projective degree.

A major breakthrough was obtained by Lang, Shapiro and Shustin \cite{LangShapiroShustin21}. Using Khovanskii's theory of simple Pfaffian manifolds \cite{Khovanskii91}, they transformed the problem of counting intersections of a line with the contour into counting isolated solutions of a Pfaffian system. Their arguments produced explicit universal upper bounds depending only on the degree of the defining polynomial. These estimates are remarkable because they are completely coefficient-independent and apply to arbitrary hypersurfaces of a fixed degree.

Although universal, the Lang--Shapiro--Shustin estimates do not exploit the sparsity of Laurent polynomials. In many situations the Newton polytope is much smaller than the simplex determined by the total degree, suggesting that considerably sharper estimates should be possible. The natural tool for exploiting sparsity is Bernstein's theorem \cite{Bernstein75}, which computes the number of isolated solutions of a nondegenerate polynomial system from the mixed volumes of its Newton polytopes. Together with the convex-geometric techniques developed in \cite{GKZ94,MaclaganSturmfels15}, Bernstein's theorem provides a refined alternative to total-degree arguments.

For plane curves these ideas are particularly transparent. The Newton polygon governs numerous geometric invariants, including the genus of a generic curve, tropicalizations and logarithmic Gauss maps. Character equations arising from line tests possess Newton polygons that are lattice segments, and the associated mixed areas frequently give much smaller estimates than B\'ezout's theorem. This interaction between amoeba geometry and Newton polytopes motivates the sparse approach developed later in the paper.

The objective of the present work is not to develop a new theory of amoebas or to extend the hypersurface theory to complete intersections. Instead, our goal is to obtain substantially smaller upper bounds for the real degree of amoeba contours than the existing universal estimates of Lang, Shapiro and Shustin. We preserve the general philosophy of their Pfaffian approach while replacing coarse total-degree arguments by logarithmic-conormal elimination techniques and sparse Newton-polytope computations. The resulting estimates remain universal but are considerably sharper for broad classes of Laurent polynomials, especially when the defining equations have highly sparse supports.

%%%%%%%%%%%%%%%%%%%%%%%%%%%%%%%%%%%%%%%%%%%%%%%%%%%%%%%%%%%%%%%%%%%%%%%%%

\section{A logarithmic Gauss map estimate for the real degree of the contour}

Let $H\subset(\mathbb C^*)^n$ be a reduced hypersurface defined by a Laurent polynomial $f\in\mathbb C[z_1^{\pm1},\dots,z_n^{\pm1}]$, and let $H_{\rm sm}$ denote its smooth part. The logarithmic map is
$
\Log:(\mathbb C^*)^n\to\mathbb R^n,\qquad \Log(z)=(\log|z_1|,\dots,\log|z_n|).
$
The contour of the amoeba is
$
\mathcal C\mathcal A_H=\Log(\operatorname{Crit}(\Log|_{H_{\rm sm}})).
$
On $H_{\rm sm}$ the logarithmic Gauss map is
$
\gamma_{\log}:H_{\rm sm}\to\mathbb P^{n-1}
$ % 
 We write
$
\Gamma_{\mathbb R}=\gamma_{\log}^{-1}(\mathbb RP^{n-1})
$
and
$
\Gamma=\operatorname{Crit}(\Log|_{H_{\rm sm}}).
$
Then $\Gamma\subset\Gamma_{\mathbb R}$. In the nondegenerate part of the contour, this inclusion is an equality.

\begin{definition}[Admissible logarithmic normal directions]
Let $L\subset\mathbb R^n$ be an affine line transverse to the regular part of $\mathcal C\mathcal A_H$. A real projective direction $[\xi]\in\mathbb RP^{n-1}$ is called admissible for $L$ if there exists $z\in\Gamma$ such that $\Log(z)\in L$ and $\gamma_{\log}(z)=[\xi]$. We denote the set of admissible directions by
$
\mathcal N_L
$
and set
$
C_L=\#\mathcal N_L
$
whenever this number is finite.
\end{definition}

The real degree is defined by intersecting the contour with affine real lines: 
$$\mathbb R\deg(\mathcal C\mathcal A_H)=\sup_L\#(L\cap\mathcal C\mathcal A_H),$$ 
where $L\subset\mathbb R^n$ ranges over affine lines transverse to the contour and the intersections are counted without multiplicity or with the prescribed generic multiplicity convention.

Fix such a transverse affine line $L$. A point $x\in L\cap\mathcal C\mathcal A_H$ means that there exists a critical point $z\in H_{\rm sm}$ such that $\Log(z)=x$. Since $z$ is critical, its logarithmic Gauss image lies in the real projective space: $\gamma_{\log}(z)\in\mathbb RP^{n-1}$. Moreover, the fact that $x$ lies on the particular affine line $L$ imposes a condition on the possible real logarithmic normal directions. Geometrically, the contour is locally a discriminant hypersurface in the amoeba, and its normal direction is controlled by the logarithmic Gauss map. Therefore, for a fixed line $L$, only certain real projective directions $[\xi]\in\mathbb RP^{n-1}$ can occur as admissible logarithmic normal directions at points of $L\cap\mathcal C\mathcal A_H$. The number of such possible directions is denoted by $C_L$.

The finite degree $d_\gamma$ of the logarithmic Gauss map means that, for a generic admissible projective direction $[\xi]\in\mathbb RP^{n-1}$, the fiber $\gamma_{\log}^{-1}([\xi])$ contains at most $d_\gamma$ points, counted with topological multiplicity. In algebraic situations this degree is often computable from the Newton polytope, but the paragraph deliberately treats it analytically: it only assumes that the map is proper and finite over the relevant real directions. Thus one does not use a Bernstein mixed-volume computation here; one uses the finite covering behavior of $\gamma_{\log}$.

There is also a second multiplicity, denoted $m_{\Log}$. Even if a critical point $z$ is known, different critical points may have the same logarithmic image: $\Log(z)=\Log(z')$. The number $m_{\Log}$ is a uniform generic bound for the multiplicity of the projection $\Log:\operatorname{Crit}(\Log|_H)\to\mathcal C\mathcal A_H$. Thus $m_{\Log}$ measures how many critical points can project to the same contour point. If this projection is generically one-to-one, then $m_{\Log}=1$. If several critical points have the same logarithmic image, then $m_{\Log}>1$.

Now the estimate follows by counting possible lifts of the intersection $L\cap\mathcal C\mathcal A_H$. For each point of $L\cap\mathcal C\mathcal A_H$, there are at most $m_{\Log}$ critical points above it under $\Log$. Each such critical point has a real logarithmic Gauss direction. The line $L$ allows at most $C_L$ admissible real directions. For each admissible direction, the logarithmic Gauss map has at most $d_\gamma$ preimages. Hence the number of lifted critical points is bounded by $d_\gamma C_L$. After accounting for the possible multiplicity of the logarithmic projection, one obtains $\#(L\cap\mathcal C\mathcal A_H)\leq m_{\Log}d_\gamma C_L.$ If there is a uniform constant $C$ such that $C_L\leq C$ for every transverse affine line $L$, then taking the supremum over all $L$ gives $$\mathbb R\deg(\mathcal C\mathcal A_H)\leq m_{\Log}d_\gamma C.$$

The important point is that this is an analytic estimate. It does not count solutions of a polynomial system by Newton polytopes. A Bernstein estimate would introduce equations defining the critical locus and then use mixed volumes of their Newton polytopes. Here the estimate instead uses the map $\gamma_{\log}:H_{\rm sm}\to\mathbb P^{n-1}$ as a finite analytic map. The real degree is controlled by three quantities: the finite topological degree $d_\gamma$ of the logarithmic Gauss map, the logarithmic projection multiplicity $m_{\Log}$, and the number $C_L$ of real logarithmic normal directions compatible with the test line $L$.

%%%%%%%%%%%%%%%%%%%%%%%%%%%%%%%%%%%%%%%%%%%%%%%%%%%%%%%%%%%%%%%%%%%%%%%%%%%

This proves the following proposition:

Let $H\subset(\mathbb C^*)^n$ be a reduced hypersurface and let $f$ be a defining equation satisfying the following hypotheses.
\begin{enumerate}
\item[(i)] The logarithmic Gauss map $\gamma_{\log}:H_{\rm sm}\to\mathbb P^{n-1}$ is proper and finite over a real Zariski open subset $U_{\mathbb R}\subset\mathbb RP^{n-1}$ containing all admissible directions which occur for the transverse affine lines under consideration. Its topological degree over $U_{\mathbb R}$ is $d_\gamma$, meaning that for every regular value $[\xi]\in U_{\mathbb R}$ the fiber $\gamma_{\log}^{-1}([\xi])$ consists of at most $d_\gamma$ points, counted with local topological multiplicity.

\item[(ii)] The critical locus $\Gamma=\operatorname{Crit}(\Log|_{H_{\rm sm}})$ maps generically finitely to its image $\mathcal C\mathcal A_H$ under $\Log$, and there is an integer $m_{\Log}\geq1$ such that for a generic point $x$ of the regular part of $\mathcal C\mathcal A_H$ one has
$
\#(\Gamma\cap\Log^{-1}(x))\leq m_{\Log}.
$

\item[(iii)] There exists a constant $C$ such that for every affine line $L\subset\mathbb R^n$ transverse to the regular part of $\mathcal C\mathcal A_H$, the set $\mathcal N_L$ of admissible logarithmic normal directions is finite and satisfies
$
C_L=\#\mathcal N_L\leq C.
$
\end{enumerate}

\begin{proposition}
Let $H\subset(\mathbb C^*)^n$ be a reduced hypersurface and let $f$ be a defining equation. Assume that the following hypotheses hold
Then  every such transverse affine line satisfies
$
\#(L\cap\mathcal C\mathcal A_H)\leq m_{\Log}d_\gamma C_L.
$
In particular,
$
\mathbb R\deg(\mathcal C\mathcal A_H)\leq m_{\Log}d_\gamma C.
$
\end{proposition}

%%%%%%%%%%%%%%%%%%%%%%%%%%%%%%%%%%%%%%%%%%%%%%%%%%%%%%%%%%%%%%%%%%%%%%%%%%%%%
%%%%%%%%%%%%%%%%%%%%%%%%%%%%%%%%%%%%%%%%%%%%%%%%%%%%%%%%%%%%%%%%%%%%%%%%%%%%%

\section{Algebraic Incidence Degree}

Let $C\subset(\C^\ast)^2$ be a smooth algebraic curve defined by a Laurent polynomial\\ $f(z,w)=\sum_{\alpha\in A}c_\alpha z^{\alpha_1}w^{\alpha_2}$, and let $\Delta=\Newt(f)=\Conv(A)\subset\R^2$. The contour of the amoeba is $\CA_C=\Log(\Crit(\Log|_C))$. Let $L=L_{\gamma,m}$ be a real affine line in the logarithmic plane, where $m=(a,b)\in\Z^2$ is primitive and $L_{\gamma,m}=\{(x,y)\in\R^2\mid ax+by=\gamma\}$. The purpose is to replace a vague constant $C$ by a precise algebraic incidence degree attached to $L$ and to compute this degree from $\Delta$ by Bernstein's theorem.
 A smooth point $(z,w)\in C$ belongs to $\Crit(\Log|_C)$ if and only if $\gamma_{\log}(z,w)\in\RP^1$, equivalently $zf_z$ and $wf_w$ are real proportional. Algebraically, this real-proportionality condition is complexified by introducing a parameter $\lambda\in\C^\ast$ and imposing $zf_z-\lambda wf_w=0$. The condition $\Log(z,w)\in L_{\gamma,m}$ is $a\log|z|+b\log|w|=\gamma$, and its algebraic complexification is the binomial character equation $z^a w^b=\eta$, where $\eta\in\C^\ast$ is a generic parameter corresponding to the translate of the affine line.
Define the complex Gauss-direction incidence variety associated with $L$ by
$$
\mathfrak I_L(f)=\{(z,w,\lambda)\in(\C^\ast)^2\times\C^\ast\mid f(z,w)=0,\ zf_z-\lambda wf_w=0,\ z^a w^b-\eta=0\}.
$$
The algebraic incidence degree is
$
\operatorname{deg}_{\textrm{inc}}(L,f)=\#\mathfrak I_L(f),
$
where the number of points is counted with multiplicities for generic $\eta$. Every transverse real point of $L\cap\CA_C$ is represented by a point of this complex incidence variety with $\lambda\in\R^\ast$ and with the appropriate real modulus condition, and therefore
$$
\#(L\cap\CA_C)\leq \operatorname{deg}_{{\textrm{inc}},{\R}}(L,f)\leq \operatorname{deg}_{\textrm{inc}}(L,f).
$$
If one uses the two-sided real contour convention in which each tropical edge intersection may give two real contour branches, the constant appearing in the refined real estimate is
$$
C_L(f)=2\,\operatorname{deg}_{\textrm{inc}}(L,f).
$$

We now compute $\operatorname{deg}_{\textrm{inc}}(L,f)$ by Bernstein's theorem. The unknowns are $(z,w,\lambda)\in(\C^\ast)^3$. The Newton polytope of $f$ is $\Delta_1=\Delta\times\{0\}\subset\R^3$. The polynomial $zf_z-\lambda wf_w$ has one part supported in $\Delta\times\{0\}$ and the other part supported in $\Delta\times\{1\}$, and hence its Newton polytope is contained in $\Delta_2=\Delta\times[0,1]$; for generic coefficients with all relevant monomials present, this containment is equality. The binomial $z^a w^b-\eta$ has Newton polytope $\Delta_3=\Conv\{(0,0,0),(a,b,0)\}$. Bernstein's theorem gives
$$
\operatorname{deg}_{\textrm{inc}}(L,f)\leq 3!\,\MV_3\left(\Delta\times\{0\},\Delta\times[0,1],\Conv\{(0,0,0),(a,b,0)\}\right).
$$
If the three equations are nondegenerate with respect to their Newton polytopes, then equality holds:
$$
\operatorname{deg}_{\textrm{inc}}(L,f)=3!\,\MV_3\left(\Delta\times\{0\},\Delta\times[0,1],\Conv\{(0,0,0),(a,b,0)\}\right).
$$

This mixed volume has a simple geometric interpretation. Since $\Delta\times\{0\}$ and $\Conv\{(0,0,0),(a,b,0)\}$ lie in the horizontal plane, while $\Delta\times[0,1]$ has vertical height one, one obtains
$$
3!\,\MV_3\left(\Delta\times\{0\},\Delta\times[0,1],\Conv\{(0,0,0),(a,b,0)\}\right)=2!\,\MV_2(\Delta,S_m),
$$
where $S_m=\Conv\{(0,0),(a,b)\}$. The mixed area of a polygon with a primitive segment is the lattice width of the polygon in the direction perpendicular to that segment. 
Let's state the following definition first:

\begin{definition}[Lattice width]
Let
$
\Delta\subset\mathbb R^n
$
be a nonempty convex lattice polytope, and let
$
m\in\mathbb Z^n\setminus\{0\}
$
be a nonzero lattice vector. The \emph{lattice width of $\Delta$ in the direction $m$} is
$$
\operatorname{width}_m(\Delta)
=
\max_{u\in\Delta}\langle m,u\rangle
-
\min_{u\in\Delta}\langle m,u\rangle,
$$
where
$
\langle m,u\rangle=m_1u_1+\cdots+m_nu_n
$
is the standard pairing between $\mathbb Z^n$ and $\mathbb R^n$. Since the function $u\mapsto\langle m,u\rangle$ is linear, the maximum and the minimum are attained at vertices of $\Delta$. Hence, if $\Delta$ is a lattice polytope, then equivalently
$$
\operatorname{width}_m(\Delta)
=
\max_{v\in\operatorname{Vert}(\Delta)}\langle m,v\rangle
-
\min_{v\in\operatorname{Vert}(\Delta)}\langle m,v\rangle.
$$
If $m$ is primitive, namely
$
\gcd(m_1,\ldots,m_n)=1,
$
then $\operatorname{width}_m(\Delta)$ is the lattice distance between the two supporting hyperplanes
$
\{u\in\mathbb R^n:\langle m,u\rangle=\min_\Delta\langle m,\cdot\rangle\}
$
and
$
\{u\in\mathbb R^n:\langle m,u\rangle=\max_\Delta\langle m,\cdot\rangle\}.
$
For a non-primitive vector $m=q m_0$, where $q\in\mathbb Z_{>0}$ and $m_0$ is primitive, one has
$$
\operatorname{width}_m(\Delta)=q\,\operatorname{width}_{m_0}(\Delta).
$$
Thus the primitive direction gives the intrinsic lattice width, while a non-primitive vector records the same direction with multiplicity.
\end{definition}

%%%%%%%%%%%%%%%%%%%%%%%%%%%%%%%%%%%%%%%%%%%%%%%%%%%%%%%%%%%%%%%%%%%%%%%%%%%%%

\section{Correction terms}

Before the refined directional estimate is stated, we define in this section the correction terms 
$
E_m(\tau,s,\gamma_{\log}).
$ 
We use the following conventions. Let
$
C=\{f=0\}\subset(\C^\ast)^2
$
be a smooth real algebraic curve with Newton polygon
$
\Delta=\Newt(f),
$
let
$
\Gamma=\Trop(C)
$
be its tropical spine, let
$
\tau
$
be the subdivision of $\Delta$ dual to $\Gamma$, and let
$
s
$
denote the Viro sign distribution on the lattice points of the subdivision. Let
$
m\in\Z^2
$
be the primitive normal direction of the tested affine logarithmic line.

\medskip

The \emph{vertex correction in direction $m$} is the nonnegative integer
$
E_m^{\mathrm{vert}}(\tau,s)
$
which measures the excess contribution to the real contour degree coming from neighborhoods of vertices of the tropical spine
$
\Gamma=\Trop(C)
$
beyond the contribution predicted by stable intersections with the edges of $\Gamma$.
More precisely, for a vertex
$
v\in\Gamma^{(0)},
$
let
$
\sigma_v\in\tau
$
be the dual two-dimensional cell of the Newton subdivision. The local patchworking chart associated with
$
\sigma_v
$
determines the real branches of the curve near the tropical vertex $v$. The stable edge contribution counts only the branches that pass through the vertex neighborhood from one adjacent edge of $\Gamma$ to another in the standard transverse way. The vertex correction
$
E_m(v;\tau,s)
$
is the number of additional intersections with the tested direction $m$ produced by local features of this chart which are not accounted for by the edge contribution. These features include local ovals contained in the vertex chart, returning branches which enter and leave the same side of the chart, and non-unimodular local cells whose real patchworking has more branches than the primitive binomial model.

The total vertex correction is
$$
E_m^{\mathrm{vert}}(\tau,s)
=
\sum_{v\in\Gamma^{(0)}}E_m(v;\tau,s).
$$
Equivalently, if the local correction at $v$ is written in terms of the primitive outgoing edge directions
$
u_1,\ldots,u_q\in\Z^2,
$
their lattice weights
$
\ell_1,\ldots,\ell_q,
$
and the Viro sign distribution $s$ on the dual cell $\sigma_v$, then
$
E_m(v;\tau,s)
$
is the difference between the actual number of real contour branches crossing  
an $m$-test line sufficiently close to $v$, restricted to a sufficiently small neighborhood of $v$, %sufficiently small $m$-test line 
in the local chart and the number predicted by the weighted stable edge contribution
$\di
\sum_{i=1}^{q}\ell_i|\det(m,u_i)|.
$
Thus
$$
E_m(v;\tau,s)
=
\Big(
N_m^{\mathrm{loc}}(v;s)
-
\sum_{i=1}^{q}\ell_i|\det(m,u_i)|
\Big)_+,
$$
where
$
N_m^{\mathrm{loc}}(v;s)
$
denotes the actual local real contour intersection number in the Viro chart and
$
(x)_+=\max\{x,0\}.
$

\medskip

Let
$
\gamma_{\log}:C\longrightarrow\mathbb P^1
$
be the logarithmic Gauss map.
 The \emph{logarithmic Gauss ramification divisor} is the ramification divisor
$
\Ram(\gamma_{\log})
$
of this map. A point
$
p\in C
$
belongs to
$
\Ram(\gamma_{\log})
$
if the differential
$
d\gamma_{\log}(p)
$
vanishes, or equivalently if the logarithmic Gauss map fails to be locally a covering map at $p$.\\

For a primitive tested direction
$
m\in\Z^2,
$
the \emph{directional logarithmic Gauss ramification correction} is the integer
$
R_m(\gamma_{\log})
$
defined as the total ramification multiplicity of those points of
$
\Ram(\gamma_{\log})
$
whose logarithmic Gauss direction contributes to the real incidence associated with the affine logarithmic lines normal to $m$. Equivalently,
$$
R_m(\gamma_{\log})
=
\sum_{p\in\Ram(\gamma_{\log})\cap \gamma_{\log}^{-1}(\mathbb RP^1_m)}
\operatorname{ram}_p(\gamma_{\log}),
$$
where
$
\operatorname{ram}_p(\gamma_{\log})
$
is the local ramification index minus one, and
$
\mathbb RP^1_m
$
denotes the real projective Gauss directions relevant to the tested normal direction $m$.

In estimates, the term
$
R_m(\gamma_{\log})
$
accounts for the failure of the real logarithmic Gauss incidence to be locally transverse. If the tested real Gauss direction avoids the branch locus of
$
\gamma_{\log},
$
then
$
R_m(\gamma_{\log})=0.
$
In all cases one has the uniform bound
$$
R_m(\gamma_{\log})
\leq
\deg \Ram(\gamma_{\log}).
$$
When $C$ is nondegenerate with Newton polygon $\Delta$, the Riemann--Hurwitz theorem gives
$$
\deg \Ram(\gamma_{\log})
=
2\deg(\gamma_{\log})+2g(C)-2.
$$
Using
$
\deg(\gamma_{\log})=2\operatorname{Area}(\Delta)
$
and
$
g(C)=\#(\operatorname{Int}\Delta\cap\Z^2),
$
one obtains
$$
\deg \Ram(\gamma_{\log})
=
4\operatorname{Area}(\Delta)
+
2\#(\operatorname{Int}\Delta\cap\Z^2)
-
2.
$$

\medskip

Let
$
\tau
$
be a regular lattice subdivision of the Newton polygon
$
\Delta,
$
and let
$
s:\Delta\cap\Z^2\longrightarrow\{\pm1\}
$
be a Viro sign distribution. The patchworking
$
(\tau,s)
$
is called \emph{primitive} if every two-dimensional cell of
$
\tau
$
is a primitive lattice triangle, meaning that each triangle has normalized area
$
1
$
or, equivalently, Euclidean area
$
\frac12.
$
Thus every local polynomial chart associated with a two-dimensional cell is equivalent, after a monomial change of coordinates and multiplication by a monomial, to a primitive trinomial model.

The patchworking
$
(\tau,s)
$
is called \emph{maximal} if the real curve produced by Viro patchworking has the maximal number of connected components allowed by Harnack's inequality for the Newton polygon
$
\Delta.
$
Equivalently, for a nonsingular real curve with Newton polygon
$
\Delta,
$
maximality means that the real part has
$
g(C)+1
$
connected components, where
$
g(C)=\#(\operatorname{Int}\Delta\cap\Z^2)
$
is the genus of a nondegenerate curve with Newton polygon $\Delta$.

\medskip

Thus \emph{the patchworking is primitive maximal} means that
$
\tau
$
is a unimodular triangulation of
$
\Delta
$
and the sign distribution
$
s
$
produces a maximal real curve. In this situation the local Viro charts are standard primitive charts. Consequently there are no non-unimodular local cells, no extra local ovals contained inside a vertex chart, and no returning branches beyond those predicted by the stable edge model. Therefore, in the primitive maximal situation,
$
E_m^{\mathrm{vert}}(\tau,s)=0
$
for every tested direction $m$, provided the tested line is generic with respect to the tropical spine.

\medskip

\begin{definition}[Total correction term]
The total correction term appearing in the refined directional contour estimate is
$$
E_m(\tau,s,\gamma_{\log})
=
E_m^{\mathrm{vert}}(\tau,s)
+
R_m(\gamma_{\log}).
$$
Here
$
E_m^{\mathrm{vert}}(\tau,s)
$
records local tropical and patchworking corrections near vertices of the spine, while
$
R_m(\gamma_{\log})
$
records the correction caused by logarithmic Gauss ramification in the tested direction. Thus
$
E_m(\tau,s,\gamma_{\log})=0
$
whenever the patchworking is primitive maximal and the logarithmic Gauss map is unramified in the tested real direction.
\end{definition}

%%%%%%%%%%%%%%%%%%%%%%%%%%%%%%%%%%%%%%%%%%%%%%%%%%%%%%%%%%%%%%%%%%%%%%%%%%%%%
\medskip

With the convention that the testing line has normal vector $m=(a,b)$, this mixed-area expression is the lattice width in the normal direction $m$:
$
2!\,\MV_2(\Delta,S_m)=\width_{m^\perp}(\Delta).
$
Consequently,
$
\operatorname{deg}_{\textrm{inc}}(L,f)\leq \width_{m^\perp}(\Delta),
$
and in the nondegenerate case equality holds:
$
\operatorname{deg}_{\textrm{inc}}(L,f)=\width_{m^\perp}(\Delta).
$
Thus the precise replacement for the constant $C$ is
$
C_L(f)=2\,\operatorname{deg}_{\textrm{inc}}(L,f),
$
and Bernstein's theorem gives
$
C_L(f)\leq 2\,\width_{m^\perp}(\Delta).
$
Therefore the refined directional estimate becomes
$$
\mathbb R\deg_m(\CA_C)\leq C_L(f)+E_m(\tau,s,\gamma_{\log})\leq 2\,\width_{m^\perp}(\Delta)+E_m(\tau,s,\gamma_{\log}),
$$
where
$$
E_m(\tau,s,\gamma_{\log})
=
E_m^{\mathrm{vert}}(\tau,s)+R_m(\gamma_{\log})
$$
records vertex corrections and logarithmic Gauss ramification. If the patchworking is primitive maximal and the logarithmic Gauss ramification correction vanishes in the tested direction, then
$$
\mathbb R\deg_m(\CA_C)\leq 2\,\width_{m^\perp}(\Delta).
$$

For a generic line one has $\Delta=\Delta_2=\Conv\{(0,0),(1,0),(0,1)\}$. In the direction $m=(1,1)$, $\width_{(1,-1)}(\Delta_2)=2$, hence $C_L(f)\leq4$. The contour of a generic line satisfies $\mathbb R\deg(\CA_C)=4$, so the incidence-degree replacement passes the generic-line test.

For $\Delta\subset d\Delta_2$ and $m=(1,1)$, one has $\width_{(1,-1)}(d\Delta_2)=2d$, hence $C_L(f)\leq4d$. This is the refined directional incidence bound. It is much sharper than the universal fewnomial bound, but it is conditional on tropical spine-control and the absence of additional ramification corrections in the tested direction.

%%%%%%%%%%%%%%%%%%%%%%%%%%%%%%%%%%%%%%%%%%%%%%%%%%%%%%%%%%%%%%%%%%%%%%%%%%%%%%

\bigskip

\section{Eliminating $\lambda$ and Computing the Newton Polygons}
Let $C=\{f(z,w)=0\}\subset(\C^\ast)^2$ be a smooth algebraic curve with Newton polygon $\Delta=\Newt(f)\subset\R^2$. Let $m=(a,b)\in\Z^2$ be primitive, and let $L_{\gamma,m}=\{(x,y)\in\R^2\mid ax+by=\gamma\}$ be an affine line in the logarithmic plane. The algebraic character associated with this affine direction is $\chi_m(z,w)=z^a w^b$. After complexifying the condition $\Log(z,w)\in L_{\gamma,m}$, one obtains the binomial equation $z^a w^b=\eta$, where $\eta\in\C^\ast$ is a generic complex parameter.

The incidence system with the auxiliary Gauss-direction parameter $\lambda$ is
$$
f(z,w)=0,\qquad zf_z-\lambda wf_w=0,\qquad z^a w^b-\eta=0.
$$
If $\lambda$ is eliminated, then the equation $zf_z-\lambda wf_w=0$ imposes no further algebraic condition on $(z,w)$ away from the exceptional locus where $wf_w=0$ and $zf_z\neq0$, because one can solve uniquely $\lambda=\frac{zf_z}{wf_w}$. Thus, after saturation by the exceptional factors, the projection to the $(z,w)$-torus is controled by the two-equation system
$
f(z,w)=0,\, z^a w^b-\eta=0.
$
This system counts the algebraic intersection of $C$ with the torus character hypersurface $\chi_m=\eta$. It does not by itself impose the real condition $\dfrac{zf_z}{wf_w}\in\mathbb R$, which is the condition defining the real logarithmic Gauss direction. The real Gauss condition is an additional real selection inside the algebraic incidence count.

The Newton polygon of the first equation is $\Delta$. The Newton polygon of the binomial equation is the segment
$
S_m=\Conv\{(0,0),(a,b)\}.
$
If $a$ or $b$ is negative, one may multiply the binomial by a Laurent monomial to place its exponents in the nonnegative quadrant. This only translates $S_m$ and does not change the mixed volume. Therefore the Newton segment is always $S_m$, up to translation.
By Bernstein's theorem, the number of isolated solutions of
$$
f(z,w)=0,\qquad z^a w^b-\eta=0
$$
in $(\C^\ast)^2$, counted with multiplicities and for generic coefficients and generic $\eta$, is
$$
N_m(\Delta)=2!\,\MV_2(\Delta,S_m).
$$
Since $S_m$ is a primitive segment, this mixed volume is the lattice width of $\Delta$ in the primitive direction perpendicular to $m$. Write
$
m^\perp=(-b,a).
$
Then
$$
N_m(\Delta)=2!\,\MV_2(\Delta,S_m)=\width_{m^\perp}(\Delta).
$$
The perpendicular direction appears because the exponent segment of the binomial is parallel to $m$, while the affine logarithmic line $ax+by=\gamma$ is perpendicular to $m$.

%%%%%%%%%%%%%%%%%%%%%%%%%%%%%%%%%%%%%%%%%%%%%%%%%%%%%%%%%%%%%%%%%%%%%%%%%%%%%%

One writes a binomial character equation
$
z^a w^b=\eta.
$
The vector
$
m=(a,b)
$
is now the exponent vector of the binomial, not the normal vector used in the width $\width_m(\Delta)$ for the logarithmic affine line. The tropicalization of the binomial equation $z^a w^b=\eta$ is
$
ax+by=\log|\eta|.
$
This affine line has normal vector $(a,b)$ and direction perpendicular to $(a,b)$. When Bernstein's theorem computes the number of solutions of
$
f(z,w)=0,\, z^a w^b=\eta,
$
it uses the Newton segment
$
S_m=\Conv\{(0,0),(a,b)\}.
$
The mixed area
$
2!\,\MV_2(\Delta,S_m)
$
is the width of $\Delta$ in the direction perpendicular to $S_m$. If
$
m=(a,b),
$
then a primitive perpendicular vector is
$
m^\perp=(-b,a).
$
Hence
$
2!\,\MV_2(\Delta,S_m)=\width_{m^\perp}(\Delta).
$

Now take
$
m=(1,-1)
$
as a binomial exponent vector. Then
$
z^a w^b=z w^{-1}=\frac{z}{w}.
$
Thus the binomial equation is
$
\frac{z}{w}=\eta,
$
or equivalently
$
z=\eta w.
$
The Newton segment has direction $(1,-1)$, and the perpendicular vector is
$
m^\perp=(1,1).
$
Therefore the Bernstein intersection number is
$
N_{(1,-1)}(d\Delta_2)=\width_{(1,1)}(d\Delta_2).
$
Since
$
\langle (1,1),(0,0)\rangle=0,\,
\langle (1,1),(d,0)\rangle=d,\,
\langle (1,1),(0,d)\rangle=d,
$
we get
$
\width_{(1,1)}(d\Delta_2)=d-0=d.
$
Thus
$
N_{(1,-1)}(d\Delta_2)=d.
$
For $d=1$, this says that a generic line meets the binomial curve $z/w=\eta$ in one point. %  

 %%%%%%%%%%%%%%%%%%%%%%%%%%%%%%%%%%%%%%%%%%%%%%%%%%%%%%%%%%%%%%%%%%%%%%%%%%%% 

\subsection*{The Triangle $\Delta=d\Delta_2$}
Let
$
\Delta=d\Delta_2=\Conv\{(0,0),(d,0),(0,d)\}.
$
For any vector $n=(u,v)$,
$$
\width_n(d\Delta_2)=d\bigl(\max\{0,u,v\}-\min\{0,u,v\}\bigr).
$$
Taking $n=m^\perp=(-b,a)$ gives
$
N_m(d\Delta_2)=d\bigl(\max\{0,-b,a\}-\min\{0,-b,a\}\bigr).
$
Equivalently,
$$
2!\,\MV_2(d\Delta_2,S_m)=d\bigl(\max\{0,-b,a\}-\min\{0,-b,a\}\bigr).
$$
For example, if $m=(1,-1)$, then $m^\perp=(1,1)$, and
$N_{(1,-1)}(d\Delta_2)=\width_{(1,1)}(d\Delta_2)=d.$
This agrees with the direct algebraic fact that a generic degree-$d$ curve meets a generic character equation $z/w=\eta$ in $d$ points in $(\C^\ast)^2$.

\subsection*{Rectangles}
Let
$
\Delta=[0,A]\times[0,B].
$
For any vector $n=(u,v)$,
$
\width_n(\Delta)=A|u|+B|v|.
$
Taking $n=m^\perp=(-b,a)$ gives
$
N_m(\Delta)=\width_{m^\perp}(\Delta)=A|b|+B|a|.
$
Therefore
$$
2!\,\MV_2([0,A]\times[0,B],S_m)=A|b|+B|a|.
$$
If $m=(1,0)$, then $m^\perp=(0,1)$ and $N_{(1,0)}(\Delta)=B$. This says that a bidegree $(A,B)$ curve meets a generic fiber $z=\eta$ in $B$ points. If $m=(0,1)$, then $m^\perp=(-1,0)$ and $N_{(0,1)}(\Delta)=A$. If $m=(1,-1)$, then $m^\perp=(1,1)$ and $N_{(1,-1)}(\Delta)=A+B$.

\subsection*{Hirzebruch Trapezoids}
Let
$
\Delta_{A,B,r}=\Conv\{(0,0),(A,0),(A+rB,B),(0,B)\}.
$
For a vector $n=(u,v)$, the width is obtained by evaluating $uX+vY$ at the four vertices:
$$
\width_n(\Delta_{A,B,r})=\max\{0,Au,(A+rB)u+Bv,Bv\}-\min\{0,Au,(A+rB)u+Bv,Bv\}.
$$
Taking $n=m^\perp=(-b,a)$ gives
$$
N_m(\Delta_{A,B,r})
=
\max\{0,-Ab,-(A+rB)b+Ba,Ba\}
-
\min\{0,-Ab,-(A+rB)b+Ba,Ba\}.
$$
Therefore
$$
2!\,\MV_2(\Delta_{A,B,r},S_m)
=
\max\{0,-Ab,-(A+rB)b+Ba,Ba\}
-
\min\{0,-Ab,-(A+rB)b+Ba,Ba\}.
$$
For the principal character directions, this gives
$$
N_{(1,0)}(\Delta_{A,B,r})=\width_{(0,1)}(\Delta_{A,B,r})=B,
$$
$$
N_{(0,1)}(\Delta_{A,B,r})=\width_{(-1,0)}(\Delta_{A,B,r})=A+rB,
$$
and
$$
N_{(1,-1)}(\Delta_{A,B,r})=\width_{(1,1)}(\Delta_{A,B,r})=A+(r+1)B.
$$

\subsection*{Resulting Incidence Constant}
After eliminating $\lambda$, the algebraic incidence degree attached to the character direction $m$ is
$$
\deg_{\rm inc}(m,\Delta)=N_m(\Delta)=2!\,\MV_2(\Delta,S_m)=\width_{m^\perp}(\Delta).
$$
If the real contour estimate uses the two-sided contour convention, the associated real incidence constant is
$$
C_m(\Delta)=2\,\deg_{\rm inc}(m,\Delta)=2\,\width_{m^\perp}(\Delta).
$$
Thus, in the absence of additional vertex and logarithmic Gauss ramification corrections, the directional estimate takes the form
$$
\mathbb R\deg_m(\CA_C)\leq 2\,\width_{m^\perp}(\Delta).
$$
With corrections included, the computable form is
$
\mathbb R\deg_m(\CA_C)\leq 2\,\width_{m^\perp}(\Delta)+E_m(\tau,s,\gamma_{\log}).
$

For $\Delta=d\Delta_2$ this becomes
$
C_m(d\Delta_2)=2d\bigl(\max\{0,-b,a\}-\min\{0,-b,a\}\bigr).
$

For $\Delta=[0,A]\times[0,B]$ this becomes
$
C_m(\Delta)=2(A|b|+B|a|).
$

For $\Delta_{A,B,r}$ this becomes
$$
C_m(\Delta_{A,B,r})=
2\left(\max\{0,-Ab,-(A+rB)b+Ba,Ba\}-\min\{0,-Ab,-(A+rB)b+Ba,Ba\}\right).
$$

%%%%%%%%%%%%%%%%%%%%%%%%%%%%%%%%%%%%%%%%%%%%%%%%%%%%%%%%%%%%%%%%%%%%%%%%%%%%
%%%%%%%%%%%%%%%%%%%%%%%%%%%%%%%%%%%%%%%%%%%%%%%%%%%%%%%%%%%%%%%%%%%%%%%%%%%%

\section{Incidence Degree with Separate Notation for Line Normals and Binomial Exponents}

Let $C=\{f(z,w)=0\}\subset(\C^\ast)^2$ be a smooth algebraic curve with Newton polygon $\Delta=\Newt(f)\subset\R^2$. We use two distinct symbols. The vector
$
n=(n_1,n_2)\in\Z^2
$
denotes the primitive normal vector of an affine line in the logarithmic plane, while
$
m=(a,b)\in\Z^2
$
denotes the exponent vector of a binomial character. These two vectors play different roles and must not be confused.

The affine logarithmic line with normal vector $n$ is
$$
L_{\gamma,n}=\{(x,y)\in\R^2\mid n_1x+n_2y=\gamma\}.
$$
The directional real degree in this direction is
$$
\mathbb R\deg_n(\CA_C)=\sup_{\gamma\in\R}\#(L_{\gamma,n}\cap\CA_C),
$$
where only transverse intersections are counted.
The binomial character with exponent vector $m=(a,b)$ is
$
\chi_m(z,w)=z^a w^b.
$
The corresponding algebraic character equation is
$$
z^a w^b=\eta,
\qquad \eta\in\C^\ast.
$$
Its Newton polygon is the lattice segment
$
S_m=\Conv\{(0,0),m\}=\Conv\{(0,0),(a,b)\}.
$
Multiplying the binomial by a Laurent monomial only translates this segment and therefore does not change its mixed volume with $\Delta$.
The real logarithmic Gauss condition is that $zf_z$ and $wf_w$ are real proportional. Equivalently, there exists $\lambda\in\R^\ast$ such that
$
zf_z=\lambda wf_w.
$
For the affine-line normal $n$, the real incidence correspondence is
$$
\mathfrak I^{\R}_{\gamma,n}(f)
=
\left\{
(z,w,\lambda)\in C\times\R^\ast\ \middle|\ zf_z=\lambda wf_w,\ n_1\log|z|+n_2\log|w|=\gamma
\right\}.
$$
This correspondence keeps the real Gauss condition and the logarithmic-line condition simultaneously. Its projection by $\Log$ maps onto $L_{\gamma,n}\cap\CA_C$, up to the usual finite ambiguity coming from distinct points of $C$ with the same logarithmic image.

To obtain an algebraic upper bound, one replaces the real logarithmic equation by a complex binomial character equation. This is where the exponent vector $m$ enters. The complex incidence correspondence associated with the binomial exponent $m$ is
$$
\mathfrak I_{\eta,m}(f)
=
\left\{
(z,w,\lambda)\in(\C^\ast)^2\times\C^\ast\ \middle|\ f(z,w)=0,\ zf_z-\lambda wf_w=0,\ z^a w^b-\eta=0
\right\}.
$$
Before eliminating $\lambda$, this system still contains the logarithmic Gauss equation
$
zf_z-\lambda wf_w=0.
$
After eliminating $\lambda$, away from the exceptional locus where $wf_w=0$ and $zf_z\neq0$, one obtains the two-equation system
$$
f(z,w)=0,\qquad z^a w^b-\eta=0.
$$
This eliminated system computes the algebraic intersection of $C$ with the torus character hypersurface $\chi_m=\eta$. It does not by itself impose the real condition $\lambda\in\R^\ast$. Therefore it gives a complex algebraic upper bound for the real incidence count, while the real Gauss condition and the possible ramification of $\gamma_{\log}$ determine which of these complex points contribute to the real contour.

By Bernstein's theorem, for generic $\eta$ and nondegenerate supports, the number of isolated solutions of
$
f(z,w)=0,\, z^a w^b-\eta=0
$
in $(\C^\ast)^2$ is
$$
\operatorname{deg}_{\textrm{inc}}(m,\Delta)=2!\MV_2(\Delta,S_m).
$$
Since $S_m$ is the segment in the exponent direction $m$, this mixed volume is the lattice width of $\Delta$ in the primitive direction perpendicular to $m$. If
$
m^\perp=(-b,a),
$
then
$$
2!\MV_2(\Delta,\Conv\{0,m\})=\width_{m^\perp}(\Delta).
$$
Thus the incidence degree associated with the binomial exponent $m$ is
$$
\operatorname{deg}_{\textrm{inc}}(m,\Delta)=\width_{m^\perp}(\Delta).
$$

Now we relate this to the contour direction. The contour direction is specified by the normal vector $n$ of the affine logarithmic line $L_{\gamma,n}$. To recover the width $\width_n(\Delta)$ from the Bernstein computation, the binomial exponent vector must be chosen perpendicular to $n$. Thus one sets
$$
m=n^\perp=(-n_2,n_1).
$$
Then
$$
m^\perp=(-m_2,m_1)=(-n_1,-n_2)=-n,
$$
and since lattice width is unchanged by changing sign,
$$
\width_{m^\perp}(\Delta)=\width_{-n}(\Delta)=\width_n(\Delta).
$$
Therefore
$$
2!\MV_2(\Delta,\Conv\{0,n^\perp\})=\width_n(\Delta).
$$
This is the precise relation between the affine-line normal direction and the binomial exponent direction.
The algebraic incidence constant for the contour direction $n$ is therefore
$$
C_n(\Delta)=2\,\operatorname{deg}_{\textrm{inc}}(n^\perp,\Delta).
$$
Using Bernstein's theorem,
$
C_n(\Delta)=2\cdot 2!\MV_2(\Delta,\Conv\{0,n^\perp\})=2\,\width_n(\Delta).
$
The factor $2$ reflects the two-sided real contour contribution near a transverse tropical edge in the spine-controlled regime.
Hence the refined directional contour estimate should be written as
$$
\mathbb R\deg_n(\CA_C)\leq 2\,\width_{n^\perp}(\Delta)+E_n(\tau,s,\gamma_{\log}),
$$
where $E_n(\tau,s,\gamma_{\log})$ contains the vertex correction and the logarithmic Gauss ramification correction. If the patchworking is primitive maximal and the tested direction avoids logarithmic Gauss ramification, then $E_n(\tau,s,\gamma_{\log})=0$ and
$$
\mathbb R\deg_n(\CA_C)\leq 2\,\width_{n^\perp}(\Delta).
$$

Let us verify the generic-line example. For a generic line,
$
\Delta=\Delta_2=\Conv\{(0,0),(1,0),(0,1)\}.
$
Take the contour direction
$
n=(1,-1).
$
Then
$$
\width_n(\Delta_2)=\width_{(1,-1)}(\Delta_2)=2.
$$
To compute this width by Bernstein's theorem, we must choose the binomial exponent
$
m=n^\perp=(1,1).
$
The binomial is
$
zw=\eta.
$
Bernstein's theorem gives
$$
2!\MV_2(\Delta_2,\Conv\{(0,0),(1,1)\})=\width_{(1,-1)}(\Delta_2)=2.
$$
Therefore
$
C_n(\Delta_2)=2\cdot2=4.
$
This agrees with the exact value
$
\mathbb R\deg(\CA_C)=4
$
for a generic line.

If one takes $m=(1,-1)$ as the binomial exponent, then the Bernstein mixed volume gives
$$
2!\MV_2(\Delta_2,\Conv\{(0,0),(1,-1)\})=\width_{(1,1)}(\Delta_2)=1.
$$
For $d\Delta_2$, it gives
$
2!\MV_2(d\Delta_2,\Conv\{(0,0),(1,-1)\})=\width_{(1,1)}(d\Delta_2)=d.
$
This is correct for the binomial equation $z/w=\eta$, but it is not the same as the contour-width direction $n=(1,-1)$. To obtain the contour width in the direction $n=(1,-1)$, one must use the perpendicular binomial exponent $m=(1,1)$.
For the triangle $d\Delta_2$, the general formula is
$$
\width_n(d\Delta_2)=d\bigl(\max\{0,n_1,n_2\}-\min\{0,n_1,n_2\}\bigr).
$$
The incidence constant in the contour direction $n$ is
$
C_n(d\Delta_2)=2d\bigl(\max\{0,n_1,n_2\}-\min\{0,n_1,n_2\}\bigr).
$
For $n=(1,-1)$, this gives
$
C_{(1,-1)}(d\Delta_2)=4d.
$

\medskip

For a rectangle
$
\Delta=[0,A]\times[0,B],
$
one has
$
\width_n(\Delta)=A|n_1|+B|n_2|.
$
Thus
$
C_n(\Delta)=2(A|n_1|+B|n_2|).
$
This is obtained by choosing the binomial exponent $m=n^\perp=(-n_2,n_1)$ and applying
$$
2!\MV_2(\Delta,\Conv\{0,m\})=\width_n(\Delta).
$$

For a Hirzebruch trapezoid
$
\Delta_{A,B,r}=\Conv\{(0,0),(A,0),(A+rB,B),(0,B)\},
$
one has
$$
\width_n(\Delta_{A,B,r})
=
\max\{0,An_1,(A+rB)n_1+Bn_2,Bn_2\}
-
\min\{0,An_1,(A+rB)n_1+Bn_2,Bn_2\}.
$$
Consequently
$$
C_n(\Delta_{A,B,r})
=
2\left(
\max\{0,An_1,(A+rB)n_1+Bn_2,Bn_2\}
-
\min\{0,An_1,(A+rB)n_1+Bn_2,Bn_2\}
\right).
$$

So, from now the symbol $n$ is always reserved for the normal direction of the affine logarithmic line and for the directional real degree $\mathbb R\deg_n(\CA_C)$. The symbol $m$ is always reserved for the exponent vector of the binomial character $z^m=\eta$. Bernstein's formula is always
$$
2!\MV_2(\Delta,\Conv\{0,m\})=\width_{m^\perp}(\Delta).
$$
To recover the contour direction $n$, one sets $m=n^\perp$, and then
$$
2!\MV_2(\Delta,\Conv\{0,n^\perp\})=\width_n(\Delta).
$$

%%%%%%%%%%%%%%%%%%%%%%%%%%%%%%%%%%%%%%%%%%%%%%%%%%%%%%%%%%%%%%%%%%%%%%%%%%%%

{ This yields the following  incidence-degree theorem.}

%%%%%%%%%%%%%%%%%%%%%%%%%%%%%%%%%%%%%%%%%%%%%%%%%%%%%%%%%%%%%%%%%%%%%%%%%%%%%
%%%%%%%%%%%%%%%%%%%%%%%%%%%%%%%%%%%%%%%%%%%%%%%%%%%%%%%%%%%%%%%%%%%%%%%%%%%%%
%%%%%%%%%%%%%%%%%%%%%%%%%%%%%%%%%%%%%%%%%%%%%%%%%%%%%%%%%%%%%%%%%%%%%%%%%%%%%

\begin{theorem}[Normal-direction incidence degree]\label{thm:normal-incidence-degree}
Let
$
C=\{f=0\}\subset(\C^\ast)^2
$
be a smooth algebraic curve with Newton polygon
$
\Delta=\operatorname{Newt}(f)\subset\R^2.
$
Let
$
n=(n_1,n_2)\in\Z^2
$
be a primitive normal direction, put
$
m=n^\perp=(-n_2,n_1),
$
and let
$
L_{\gamma,n}
=
\{x\in\R^2\mid\langle n,x\rangle=\gamma\}.
$
Assume that, for generic
$
\gamma\in\R
$
and
$
\theta\in\R/2\pi\Z,
$
the character-incidence system
$
f(z,w)=0,
$\, 
$
z^{n_1}w^{n_2}=e^{\gamma+i\theta}
$
is zero-dimensional and Bernstein nondegenerate. Then
$$
\deg_{\mathrm{inc}}(n,\Delta)
=
2!\MV_2
\left(
\Delta,\Conv\{0,n\}
\right)
=
\width_{n^\perp}(\Delta).
$$
Since
$
m=n^\perp,
$
this may be written as
$
\deg_{\mathrm{inc}}(n,\Delta)
=
\width_m(\Delta).
$
If the two-sided complex incidence constant is defined by
$
C_n^{\mathrm{two}}(\Delta)
=
2\deg_{\mathrm{inc}}(n,\Delta),
$
then
$$
C_n^{\mathrm{two}}(\Delta)
=
2\width_m(\Delta).
$$
\end{theorem}

\begin{proof}
The vector
$
n=(n_1,n_2)
$
is the primitive normal vector of the affine logarithmic line
$
L_{\gamma,n}
=
\{x\in\R^2\mid\langle n,x\rangle=\gamma\}.
$
For
$
(z,w)\in(\C^\ast)^2,
$
one has
$$
\langle n,\Log(z,w)\rangle
=
n_1\log|z|+n_2\log|w|.
$$
Using the multiplicativity of the absolute value,
$$
n_1\log|z|+n_2\log|w|
=
\log\left(|z|^{n_1}|w|^{n_2}\right)
=
\log\left|z^{n_1}w^{n_2}\right|.
$$
Therefore
$$
\Log^{-1}(L_{\gamma,n})
=
\left\{
(z,w)\in(\C^\ast)^2
\;\middle|\;
\left|z^{n_1}w^{n_2}\right|=e^\gamma
\right\}.
$$
Thus the real hypersurface $\Log^{-1}(L_{\gamma,n})$ is the union, over all phases $\theta$, of the complex algebraic character slices
$
z^{n_1}w^{n_2}=e^{\gamma+i\theta}.
$
Fix generic $\gamma$ and $\theta$, and write
$
\eta=e^{\gamma+i\theta}\in\C^\ast.
$
The corresponding incidence system is
$
f(z,w)=0,
$
$
z^{n_1}w^{n_2}-\eta=0.
$
The Newton polygon of the first equation is $\Delta$. The support of the second Laurent polynomial consists of the two exponent vectors
$
0
\, \text{and}\, 
n.
$
Hence its Newton polygon is
$
S_n=\Conv\{0,n\}.
$
If one or both coordinates of $n$ are negative, this causes no difficulty because the equation is a Laurent equation. Multiplication by a Laurent monomial translates the Newton segment, and mixed volume is invariant under translations of its arguments.

By hypothesis, the incidence system is zero-dimensional and Bernstein nondegenerate. The Bernstein--Kushnirenko theorem therefore gives the exact number of solutions in $(\C^\ast)^2$, counted with algebraic multiplicity:
$$
\#\left\{
(z,w)\in(\C^\ast)^2
\;\middle|\;
f(z,w)=0,\ 
z^{n_1}w^{n_2}=\eta
\right\}
=
2!\MV_2(\Delta,S_n).
$$
The left-hand side is, by definition, the generic complex incidence degree
$
\deg_{\mathrm{inc}}(n,\Delta).
$
Consequently,
$$
\deg_{\mathrm{inc}}(n,\Delta)
=
2!\MV_2
\left(
\Delta,\Conv\{0,n\}
\right).
$$

It remains to identify the mixed area with the appropriate lattice width. Recall that for a nonzero integral vector $u$, the lattice width of $\Delta$ in direction $u$ is
$$
\width_u(\Delta)
=
\max_{p\in\Delta}\langle u,p\rangle
-
\min_{p\in\Delta}\langle u,p\rangle.
$$
We shall prove that
$
2!\MV_2
\left(
\Delta,\Conv\{0,n\}
\right)
=
\width_{n^\perp}(\Delta).
$

Since $n$ is primitive, it can be completed to a lattice basis. Thus there exists
$
A\in\operatorname{GL}(2,\Z)
$
such that
$
An=e_1=(1,0).
$
A unimodular linear transformation preserves normalized area and mixed area. Therefore
$$
2!\MV_2
\left(
\Delta,\Conv\{0,n\}
\right)
=
2!\MV_2
\left(
A\Delta,\Conv\{0,e_1\}
\right).
$$
Put
$
P=A\Delta
$
and
$
S_0=\Conv\{0,e_1\}.
$
For
$
t\ge0,
$
consider the Minkowski sum
$
P+tS_0.
$
At each height $y$, the horizontal slice of the convex polygon $P$ is either empty or a compact interval. If
$
P_y=[\alpha(y),\beta(y)],
$
then
$$
(P+tS_0)_y
=
[\alpha(y),\beta(y)+t].
$$
Hence every nonempty horizontal slice increases in length by exactly $t$.

Let
$
y_{\min}
=
\min_{(x,y)\in P}y
$
and
$
y_{\max}
=
\max_{(x,y)\in P}y.
$
By Cavalieri's principle,
$$
\operatorname{Area}(P+tS_0)
=
\operatorname{Area}(P)
+
t(y_{\max}-y_{\min}).
$$
The difference
$
y_{\max}-y_{\min}
$
is the lattice width of $P$ in the vertical direction:
$
y_{\max}-y_{\min}
=
\width_{e_2}(P).
$

On the other hand, the polarization formula for mixed area gives
$$
\operatorname{Area}(P+tS_0)
=
\operatorname{Area}(P)
+
2t\MV_2(P,S_0)
+
t^2\operatorname{Area}(S_0).
$$
Since $S_0$ is one-dimensional,
$
\operatorname{Area}(S_0)=0.
$
Thus
$$
\operatorname{Area}(P+tS_0)
=
\operatorname{Area}(P)
+
2t\MV_2(P,S_0).
$$
Comparing the coefficients of $t$ yields
$
2\MV_2(P,S_0)
=
\width_{e_2}(P).
$
Since
$
2!=2,
$
we have
$$
2!\MV_2(P,S_0)
=
\width_{e_2}(P).
$$

We now return to the original coordinates. Under the unimodular map $A$, the covector defining the vertical coordinate pulls back to a primitive covector annihilating $n$. Such a covector is, up to sign,
$
n^\perp=(-n_2,n_1).
$
Since lattice width is unchanged when the direction is multiplied by $-1$, one obtains
$
\width_{e_2}(A\Delta)
=
\width_{n^\perp}(\Delta).
$
Therefore
$$
2!\MV_2
\left(
\Delta,\Conv\{0,n\}
\right)
=
\width_{n^\perp}(\Delta).
$$

Combining this identity with Bernstein's theorem gives
$$
\deg_{\mathrm{inc}}(n,\Delta)
=
2!\MV_2
\left(
\Delta,\Conv\{0,n\}
\right)
=
\width_{n^\perp}(\Delta).
$$
Since
$
m=n^\perp,
$
this is equivalently
$
\deg_{\mathrm{inc}}(n,\Delta)
=
\width_m(\Delta).
$

Finally, by the definition
$
C_n^{\mathrm{two}}(\Delta)
=
2\deg_{\mathrm{inc}}(n,\Delta),
$
one obtains
$
C_n^{\mathrm{two}}(\Delta)
=
2\width_m(\Delta).
$
This proves the theorem.
\end{proof}

%%%%%%%%%%%%%%%%%%%%%%%%%%%%%%%%%%%%%%%%%%%%%%%%%%%%%%%%%%%%%%%%%%%%%%%%%%%%%%%%
%%%%%%%%%%%%%%%%%%%%%%%%%%%%%%%%%%%%%%%%%%%%%%%%%%%%%%%%%%%%%%%%%%%%%%%%%%%%%%%%
%%%%%%%%%%%%%%%%%%%%%%%%%%%%%%%%%%%%%%%%%%%%%%%%%%%%%%%%%%%%%%%%%%%%%%%%%%%%%%%%

\section{The Complex Incidence Degree and the Separate Real Two-Phase Estimate}

Let
$
C=\{f=0\}\subset(\C^\ast)^2
$
be a smooth algebraic curve with Newton polygon
$
\Delta=\operatorname{Newt}(f).
$
Let
$
n=(n_1,n_2)\in\Z^2
$
be a primitive normal direction and put
$
m=n^\perp=(-n_2,n_1).
$
The affine logarithmic test family is
$$
L_{\gamma,n}
=
\{x\in\R^2\mid\langle n,x\rangle=\gamma\}.
$$
The associated Laurent character is
$
\chi_n(z,w)=z^{n_1}w^{n_2}.
$
For
$
(z,w)\in(\C^\ast)^2,
$
one has
$$
\langle n,\Log(z,w)\rangle
=
\log|\chi_n(z,w)|.
$$
Hence
$$
\Log^{-1}(L_{\gamma,n})
=
\{(z,w)\in(\C^\ast)^2\mid |\chi_n(z,w)|=e^\gamma\}.
$$
This real hypersurface is foliated by the complex character slices
$
\chi_n(z,w)=e^{\gamma+i\theta},
\, 
\theta\in\R/2\pi\Z.
$

Assume that, for generic $\gamma$ and $\theta$, the system
$
f(z,w)=0,
\,
\chi_n(z,w)=e^{\gamma+i\theta}
$
is zero-dimensional and Bernstein nondegenerate. Then the complex incidence degree is
$$
\deg_{\mathrm{inc}}(n,\Delta)
=
2!\MV_2
\left(
\Delta,\Conv\{0,n\}
\right)
=
\width_m(\Delta).
$$
The two-sided complex incidence constant is defined by
$
C_n^{\mathrm{two}}(\Delta)
=
2\deg_{\mathrm{inc}}(n,\Delta).
$
Therefore
$$
C_n^{\mathrm{two}}(\Delta)
=
2\width_m(\Delta).
$$

This complex algebraic identity does not by itself imply a real contour estimate. The real bound requires an additional hypothesis controlling the phases of the logarithmic critical points lying over the modulus level.

Let
$
\gamma_{\log}:C\longrightarrow\PP^1
$
be the logarithmic Gauss map. The logarithmic critical locus is
$
\Crit(\Log|_C)
=
\gamma_{\log}^{-1}(\RP^1).
$
For generic $\gamma$, define the finite real critical incidence set
$$
\mathcal K_{\gamma,n}
=
\Crit(\Log|_C)
\cap
\Log^{-1}(L_{\gamma,n}).
$$
Every point of
$
L_{\gamma,n}\cap\CA_C
$
has at least one lift in $\mathcal K_{\gamma,n}$.

\begin{definition}[Real two-phase control in normal direction $n$]
The contour satisfies the real two-phase control hypothesis in the primitive normal direction $n$ if, for every generic $\gamma$, there exist phases
$
\theta_1(\gamma),
\, 
\theta_2(\gamma)
\in
\R/2\pi\Z
$
such that every point
$
p\in\mathcal K_{\gamma,n}
$
satisfies
$
\chi_n(p)
=
e^{\gamma+i\theta_j(\gamma)}
$
for at least one
$
j\in\{1,2\}.
$
Equivalently,
$$
\mathcal K_{\gamma,n}
\subset
\{\chi_n=e^{\gamma+i\theta_1(\gamma)}\}
\cup
\{\chi_n=e^{\gamma+i\theta_2(\gamma)}\}.
$$
We also require that, for generic $\gamma$, each of the two systems
$
f(z,w)=0,
\, 
\chi_n(z,w)=e^{\gamma+i\theta_j(\gamma)}
$
is zero-dimensional and Bernstein nondegenerate.
\end{definition}

\begin{proposition}[Real two-phase incidence bound]\label{prop:real-two-phase-normal}
Let
$
C=\{f=0\}\subset(\C^\ast)^2
$
be a smooth algebraic curve with Newton polygon $\Delta$, and let $n\in\Z^2$ be a primitive normal direction. Put
$
m=n^\perp.
$
Assume that the generic character-incidence systems are zero-dimensional and Bernstein nondegenerate, so that
$
\deg_{\mathrm{inc}}(n,\Delta)
=
\width_m(\Delta).
$
Assume moreover that the contour satisfies the real two-phase control hypothesis in direction $n$. Then, for every generic $\gamma$,
$$
\#\left(
L_{\gamma,n}\cap\CA_C
\right)
\le
2\deg_{\mathrm{inc}}(n,\Delta).
$$
Consequently,
$$
\Rdeg_n(\CA_C)
\le
C_n^{\mathrm{two}}(\Delta)
=
2\width_m(\Delta).
$$
\end{proposition}

\begin{proof}
Fix a generic value
$
\gamma\in\R.
$
Let
$
x\in L_{\gamma,n}\cap\CA_C.
$
By the definition of the amoeba contour, there exists at least one point
$
p\in C
$
such that
$
\Log(p)=x
$
and
$
\gamma_{\log}(p)\in\RP^1.
$
Since
$
x\in L_{\gamma,n},
$
one has
$
\langle n,\Log(p)\rangle=\gamma.
$
Equivalently,
$
|\chi_n(p)|=e^\gamma.
$
Therefore
$$
p\in
\Crit(\Log|_C)
\cap
\Log^{-1}(L_{\gamma,n})
=
\mathcal K_{\gamma,n}.
$$

By the real two-phase control hypothesis, there exists
$
j\in\{1,2\}
$
such that
$
\chi_n(p)
=
e^{\gamma+i\theta_j(\gamma)}.
$
Thus every contour image point has at least one lift in the union
$
\mathcal I_{\gamma,n}^{(1)}
\cup
\mathcal I_{\gamma,n}^{(2)},
$
where
$$
\mathcal I_{\gamma,n}^{(j)}
=
\left\{
(z,w)\in(\C^\ast)^2
\;\middle|\;
f(z,w)=0,\ 
\chi_n(z,w)=e^{\gamma+i\theta_j(\gamma)}
\right\}.
$$

The logarithmic map may identify distinct points of this union. Consequently, the number of distinct logarithmic image points is bounded by the number of lifts:
$$
\#\left(
L_{\gamma,n}\cap\CA_C
\right)
\le
\#
\left(
\mathcal I_{\gamma,n}^{(1)}
\cup
\mathcal I_{\gamma,n}^{(2)}
\right).
$$
The cardinality of a union is bounded by the sum of the cardinalities, hence
$$
\#\left(
L_{\gamma,n}\cap\CA_C
\right)
\le
\#
\mathcal I_{\gamma,n}^{(1)}
+
\#
\mathcal I_{\gamma,n}^{(2)}.
$$

For each
$
j\in\{1,2\},
$
the Newton polygons of the two equations defining $\mathcal I_{\gamma,n}^{(j)}$ are
$
\Delta
$
and
$
\Conv\{0,n\}.
$
By the assumed Bernstein nondegeneracy,
$$
\#
\mathcal I_{\gamma,n}^{(j)}
=
2!\MV_2
\left(
\Delta,\Conv\{0,n\}
\right),
$$
where the solutions are counted with algebraic multiplicity. Since an ordinary cardinality is at most the corresponding multiplicity count,
$$
\#
\mathcal I_{\gamma,n}^{(j)}
\le
2!\MV_2
\left(
\Delta,\Conv\{0,n\}
\right).
$$
Using
$
2!\MV_2
\left(
\Delta,\Conv\{0,n\}
\right)
=
\width_m(\Delta),
$
one obtains
$
\#
\mathcal I_{\gamma,n}^{(j)}
\le
\width_m(\Delta).
$
Therefore
$$
\#\left(
L_{\gamma,n}\cap\CA_C
\right)
\le
2\width_m(\Delta).
$$

By definition, the directional image-counting degree is
$
\Rdeg_n(\CA_C)
=
\sup_{\gamma\ \mathrm{generic}}
\#
\left(
L_{\gamma,n}\cap\CA_C
\right).
$
Taking the supremum over generic $\gamma$ gives
$
\Rdeg_n(\CA_C)
\le
2\width_m(\Delta).
$
Finally,
$
C_n^{\mathrm{two}}(\Delta)
=
2\deg_{\mathrm{inc}}(n,\Delta)
=
2\width_m(\Delta).
$
Hence
$$
\Rdeg_n(\CA_C)
\le
C_n^{\mathrm{two}}(\Delta)
=
2\width_m(\Delta).
$$
\end{proof}

%%%%%%%%%%%%%%%%%%%%%%%%%%%%%%%%%%%%%%%%%%%%%%%%%%%%%%%%%%%%%%%%%%%%%%%%%%%%%
%%%%%%%%%%%%%%%%%%%%%%%%%%%%%%%%%%%%%%%%%%%%%%%%%%%%%%%%%%%%%%%%%%%%%%%%%%%%% 

\section{Contour Degree, Logarithmic Gauss Map, and Real Incidence Correspondence}

Let $C\subset(\C^\ast)^2$ be a smooth real algebraic curve defined by a real Laurent polynomial $f(z,w)$.  
At a smooth point $p=(z,w)\in C$, the point $p$ is critical for $\Log|_C$ if and only if the tangent line $T_pC$ is not transverse to the real torus orbit through $p$.
In coordinates, this means that $zf_z(p)$ and $wf_w(p)$ are real proportional. Thus there exists $\lambda\in\R$ such that
$
zf_z(p)=\lambda\,wf_w(p).
$
If both $zf_z(p)$ and $wf_w(p)$ are nonzero, then $\lambda\in\R^\ast$ and
$
\lambda=\dfrac{zf_z(p)}{wf_w(p)}.
$
The exceptional cases where one of the two logarithmic derivatives vanishes are treated by using the homogeneous form $[zf_z:wf_w]\in\RP^1$. 
Let $n=(n_1,n_2)\in\Z^2$ be a primitive vector. For $\gamma\in\R$, set
$
L_{\gamma,n}=\{(x,y)\in\R^2\mid n_1x+n_2y=\gamma\}.
$
The directional contour degree is
$$
\mathbb R\deg_n(\CA_C)=\sup_{\gamma\in\R}\#(L_{\gamma,n}\cap\CA_C),
$$
where only transverse intersections are counted.
The real incidence correspondence associated with the direction $n$ and the translate $\gamma$ is
$$
\mathfrak I_{\gamma,n}^{\R}(C)=
\left\{
p\in C
\ \middle|\
\gamma_{\log}(p)\in\RP^1,\quad \Log(p)\in L_{\gamma,n}
\right\}.
$$
Equivalently, away from the homogeneous exceptional points, this can be written as
$$
\mathfrak I_{\gamma,n}^{\R}(C)=
\left\{
(z,w,\lambda)\in C\times\R^\ast
\ \middle|\
zf_z=\lambda wf_w,\quad n_1\log|z|+n_2\log|w|=\gamma
\right\}.
$$
The natural projection
$
\pi_{\Log}:\mathfrak I_{\gamma,n}^{\R}(C)\longrightarrow L_{\gamma,n}\cap\CA_C,\,  p\longmapsto \Log(p),
$
is surjective by definition of the contour. Hence
$$
\#(L_{\gamma,n}\cap\CA_C)\leq \#\mathfrak I_{\gamma,n}^{\R}(C)
$$
whenever both sets are finite. If the restriction of $\Log$ to $\Crit(\Log|_C)$ is generically one-to-one over the relevant part of the contour, then equality holds for generic transverse $L_{\gamma,n}$. In general the inequality may be strict, because several points of $C$ may have the same logarithmic image.
This gives the first rigorous relation: contour intersections are logarithmic images of real logarithmic-Gauss incidence points.

We now pass from the real incidence correspondence to a complex algebraic upper bound. The real equation
$$
n_1\log|z|+n_2\log|w|=\gamma
$$
is not algebraic. To algebraize the incidence, choose an integral vector
$
m=(a,b)\in\Z^2
$
and consider the character equation
$
z^aw^b=\eta,\,  \eta\in\C^\ast.
$
The logarithmic modulus of this equation is
$$
a\log|z|+b\log|w|=\log|\eta|.
$$
Therefore, if one wants this character equation to have the same logarithmic normal direction as $L_{\gamma,n}$, one should choose $m=n$. If instead one wants Bernstein's mixed volume with the segment $\Conv\{0,m\}$ to recover $\width_n(\Delta)$, one chooses $m=n^\perp$. These are different uses of the character equation and should not be confused. For the real incidence theorem below, the algebraic character used to dominate the logarithmic slice in the same normal direction is
$
m=n.
$
Fix $\eta\in\C^\ast$. Define the complex algebraic incidence set
$$
\mathfrak I_{\eta,n}^{\C}(C)=
\left\{
(z,w,\lambda)\in(\C^\ast)^2\times\C
\ \middle|\
f(z,w)=0,\quad zf_z-\lambda wf_w=0,\quad z^{n_1}w^{n_2}-\eta=0
\right\}.
$$
This is an algebraic set in $(\C^\ast)^2\times\C$. After homogenizing the $\lambda$-coordinate, one may equivalently write the logarithmic Gauss condition as an incidence with $\CP^1$. The affine form above is sufficient on the chart where $wf_w\neq0$. The homogeneous version is
$$
\mathfrak J_{\eta,n}^{\C}(C)=
\left\{
(z,w,[\lambda_0:\lambda_1])\in C\times\CP^1
\ \middle|\
\lambda_1zf_z-\lambda_0wf_w=0,\quad z^{n_1}w^{n_2}=\eta
\right\}.
$$
This homogeneous incidence correspondence avoids the exceptional cases $wf_w=0$ and $zf_z=0$.

For generic $\eta$, the set $\mathfrak J_{\eta,n}^{\C}(C)$ is finite. Define the algebraic Gauss incidence degree in the direction $n$ by
$$
D_n(C)=\#\mathfrak J_{\eta,n}^{\C}(C),
$$
counted with scheme-theoretic multiplicities for generic $\eta$. Since the incidence condition in $\CP^1$ imposes only the tautological proportionality relation defining the logarithmic Gauss parameter, the projection to $C$ identifies $\mathfrak J_{\eta,n}^{\C}(C)$ with the finite set
$
C\cap\{z^{n_1}w^{n_2}=\eta\}.
$
Thus
$$
D_n(C)=\#\{(z,w)\in C\mid z^{n_1}w^{n_2}=\eta\},
$$
counted with multiplicities.
Let $\Delta=\Newt(f)$ and let
$
S_n=\Conv\{(0,0),n\}.
$
Bernstein's theorem gives
$$
D_n(C)\leq 2!\MV_2(\Delta,S_n),
$$
with equality if the system $f=0$, $z^{n_1}w^{n_2}-\eta=0$ is nondegenerate with respect to its Newton polytopes. Since $S_n$ is a segment, the mixed volume is a lattice width:
$$
2!\MV_2(\Delta,S_n)=\width_{n^\perp}(\Delta),
$$
where $n^\perp=(-n_2,n_1)$.
This shows precisely what Bernstein's theorem controls. It controls the algebraic degree of intersection of $C$ with the character hypersurface whose exponent vector is $n$, and this degree is the width of $\Delta$ in the direction perpendicular to $n$.

If the goal is to obtain the Newton width $\width_n(\Delta)$, then the character exponent must be chosen perpendicular to $n$. Namely, put
$
m=n^\perp.
$
Then
$$
2!\MV_2(\Delta,\Conv\{0,m\})=\width_{m^\perp}(\Delta)=\width_n(\Delta).
$$
In this convention, the character equation is
$
z^{m_1}w^{m_2}=\eta,
$
and its Bernstein degree is exactly $\width_n(\Delta)$.
The  algebraic incidence statement is therefore:
$$
D_{n^\perp}(C)\leq \width_n(\Delta),
$$
with equality under Bernstein nondegeneracy. However, this algebraic degree is not by itself equal to $\#(L_{\gamma,n}\cap\CA_C)$. It is a complex algebraic degree attached to a character slice. To pass from this degree to contour intersections one needs a real comparison hypothesis.
The comparison theorem is the following.

\medskip

\begin{theorem}
Let $C\subset(\C^\ast)^2$ be a smooth real algebraic curve, let $n\in\Z^2$ be primitive, and let $\gamma\in\R$ be such that $L_{\gamma,n}$ meets $\CA_C$ transversally in finitely many points. Then
$$
L_{\gamma,n}\cap\CA_C
=
\Log\left(
\gamma_{\log}^{-1}(\RP^1)\cap\Log^{-1}(L_{\gamma,n})
\right).
$$
Consequently,
$$
\#(L_{\gamma,n}\cap\CA_C)
\leq
\#\left(
\gamma_{\log}^{-1}(\RP^1)\cap\Log^{-1}(L_{\gamma,n})
\right).
$$
If $\Log$ is injective on this finite incidence set, then equality holds. More generally, if each point of $L_{\gamma,n}\cap\CA_C$ has at most $\mu_n$ preimages in $\gamma_{\log}^{-1}(\RP^1)\cap\Log^{-1}(L_{\gamma,n})$, then
$$
\#(L_{\gamma,n}\cap\CA_C)
\leq
\#\left(
\gamma_{\log}^{-1}(\RP^1)\cap\Log^{-1}(L_{\gamma,n})
\right)
\leq
\mu_n\,\#(L_{\gamma,n}\cap\CA_C).
$$
\end{theorem}
\medskip

\begin{proof}
By definition, $q\in L_{\gamma,n}\cap\CA_C$ if and only if there exists $p\in\Crit(\Log|_C)$ such that $\Log(p)=q$ and $q\in L_{\gamma,n}$. Since $\Crit(\Log|_C)=\gamma_{\log}^{-1}(\RP^1)$, this is equivalent to the existence of
$$
p\in\gamma_{\log}^{-1}(\RP^1)\cap\Log^{-1}(L_{\gamma,n})
$$
with $q=\Log(p)$. This proves the equality of sets after applying $\Log$. The cardinality inequality follows from the surjectivity of the map from the incidence set to the contour intersection set. The final assertion follows by bounding the size of the fibers of this map.
\end{proof}

\medskip

This theorem contains no tropical hypothesis. It tells us exactly what contour intersections are: they are logarithmic images of real logarithmic-Gauss incidence points.

To obtain a numerical upper bound from Bernstein's theorem, one needs an additional hypothesis comparing the real logarithmic incidence
$
\gamma_{\log}^{-1}(\RP^1)\cap\Log^{-1}(L_{\gamma,n})
$
with an algebraic character incidence. Such a comparison is not automatic for an arbitrary curve, because $\Log^{-1}(L_{\gamma,n})$ is the real hypersurface
$
|z|^{n_1}|w|^{n_2}=e^\gamma,
$
not the complex algebraic hypersurface
$
z^{n_1}w^{n_2}=\eta.
$
The latter fixes both modulus and argument, while the former fixes only modulus. Consequently, the Bernstein incidence degree is not a direct universal bound for the real logarithmic incidence unless one has an additional finite phase-control hypothesis.

A phase-control hypothesis can be stated as follows. For the chosen direction $n$, suppose that for every generic $\gamma$ there exists a finite set of phases
$
\Theta_{\gamma,n}\subset (S^1)^2
$
with cardinality at most $M_n$ such that every point of
$
\gamma_{\log}^{-1}(\RP^1)\cap\Log^{-1}(L_{\gamma,n})
$
lies on one of the algebraic character slices
$$
z^{m_1}w^{m_2}=\eta_\theta,
\qquad m=n^\perp,
\qquad \theta\in\Theta_{\gamma,n}.
$$
Then each such slice has at most
$
D_{n^\perp}(C)\leq\width_n(\Delta)
$
complex points by Bernstein's theorem. Therefore
$
\#(L_{\gamma,n}\cap\CA_C)
\le
M_n\,\width_{n^\perp}(\Delta).
$
If the phase-control constant is $M_n=2$, then
$
\mathbb R\deg_n(\CA_C)\leq 2\width_{n^\perp}(\Delta)
$
(see Appendix B).

In the tropical or patchworked setting this phase-control hypothesis is replaced by a geometric condition. If the contour is controlled by the tropical spine in direction $n$, then the stable tropical intersection number of $L_{\gamma,n}$ with the spine $\Gamma=\Trop(C)$ is
$
I_n(\Gamma)=\width_{n^\perp}(\Delta),
$
(see Appendix D). If each stable tropical intersection of multiplicity $r$ contributes at most $2r$ real contour intersections, and if there are no extra contributions from vertex charts or from ramification of $\gamma_{\log}$, then
$$
\#(L_{\gamma,n}\cap\CA_C)\leq2I_n(\Gamma)=2\width_{n^\perp}(\Delta).
$$
Taking the supremum over $\gamma$ gives
$
\mathbb R\deg_n(\CA_C)\leq2\width_{n^\perp}(\Delta).
$
In the presence of ramification, the corrected theorem is
$$
\mathbb R\deg_n(\CA_C)\leq2\width_{n^\perp}(\Delta)+E_n^{\rm vert}(\tau,s)+R_n(\gamma_{\log}),
$$
where
$$
R_n(\gamma_{\log})
=
\sum_{p\in\Ram(\gamma_{\log})\cap\gamma_{\log}^{-1}(\Lambda_n)}
\ord_p(\Ram(\gamma_{\log})).
$$
Here $\Lambda_n\in\RP^1$ is the logarithmic tangent direction corresponding to the tested affine-line direction. This correction is necessary because branch values of $\gamma_{\log}$ may create tangencies, multiplicity jumps, or local merging of contour branches.
The final rigorous conclusion is therefore as follows.
The identity
$$
L_{\gamma,n}\cap\CA_C
=
\Log\left(
\gamma_{\log}^{-1}(\RP^1)\cap\Log^{-1}(L_{\gamma,n})
\right)
$$
is always true. It is the exact relationship between contour intersections and real logarithmic Gauss fibers.
The Bernstein formula
$$
2!\MV_2(\Delta,\Conv\{0,n^\perp\})=\width_n(\Delta)
$$
is always the correct algebraic incidence degree after choosing the binomial exponent perpendicular to the affine-line normal.
The inequality
$$
\mathbb R\deg_n(\CA_C)\leq2\width_{n^\perp}(\Delta)
$$
is not a universal consequence of these two facts. It is a conditional theorem requiring a phase-control, spine-control, simple Harnack, or primitive patchworking hypothesis which ensures that the real logarithmic Gauss incidence contributes at most two real contour branches per unit of Bernstein incidence degree, and that vertex and ramification corrections vanish.

%%%%%%%%%%%%%%%%%%%%%%%%%%%%%%%%%%%%%%%%%%%%%%%%%%%%%%%%%%%%%%%%%%%%%%%%%%%%

%%%%%%%%%%%%%%%%%%%%%%%%%%%%%%%%%%%%%%%%%%%%%%%%%%%%%%%%%%%%%%%%%%%%%%%%%%%%%%%%

\section{Spine-Controlled Contours and Directional Real Degree}

\begin{definition}[Stable intersection number with the spine]
Let
$
\Gamma\subset\R^2
$
be a weighted tropical curve, and let
$
m\in\Z^2
$
be primitive. For a generic affine line
$
L_{\gamma,m}=\{x\in\R^2\mid \langle m,x\rangle=\gamma\},
$
the stable tropical intersection number of $L_{\gamma,m}$ with $\Gamma$ is
$
I_m(\Gamma)
=
L_{\gamma,m}\cdot_{\mathrm{st}}\Gamma.
$
Equivalently, if $L_{\gamma,m}$ meets $\Gamma$ only in the relative interiors of edges, then
$$
I_m(\Gamma)
=
\sum_{p\in L_{\gamma,m}\cap\Gamma}
w(e_p)|\det(m,u_{e_p})|,
$$
where $e_p$ is the edge of $\Gamma$ containing $p$, $u_{e_p}\in\Z^2$ is the primitive direction vector of $e_p$, and $w(e_p)$ is the tropical weight of $e_p$.
\end{definition}

\begin{theorem}[Spine control implies the two-sided contour bound]
Let
$
C\subset(\C^\ast)^2
$
be a smooth real algebraic curve. Assume that its amoeba contour
$
\CA_C
$
is spine-controlled in the following precise sense. There is a weighted tropical spine
$
\Gamma\subset\R^2
$
such that, outside pairwise disjoint sufficiently small neighborhoods of the vertices of $\Gamma$, the contour $\CA_C$ is a two-sided real analytic smoothing of $\Gamma$, and inside each vertex neighborhood no additional contour branch contributes more intersections with generic logarithmic lines than those already forced by the adjacent weighted edges of $\Gamma$. Then, for every primitive vector
$
m\in\Z^2,
$
one has
$$
\mathbb R\deg_m(\CA_C)\leq 2I_m(\Gamma).
$$
\end{theorem}

\begin{proof}
Fix a primitive vector
$
m\in\Z^2.
$
For a real number $\gamma$, write
$
L_{\gamma,m}=\{x\in\R^2\mid \langle m,x\rangle=\gamma\}.
$
It is enough to prove that, for every generic $\gamma$ for which $L_{\gamma,m}$ is transverse to $\CA_C$ and to the relevant smooth edge pieces of $\Gamma$, one has
$
\#(L_{\gamma,m}\cap\CA_C)\leq 2I_m(\Gamma).
$
Taking the supremum over such $\gamma$ then gives the desired inequality for
$
\mathbb R\deg_m(\CA_C).
$

Choose pairwise disjoint open neighborhoods
$
U_v
$
of the vertices $v\in\Gamma^{(0)}$ small enough so that each component of
$
\Gamma\setminus\bigcup_v U_v
$
is contained in the relative interior of a single edge of $\Gamma$. Put
$
U=\bigcup_v U_v.
$
By the spine-control hypothesis, on
$
\R^2\setminus U
$
the contour $\CA_C$ consists of two real analytic sheets which are small smoothings of the corresponding weighted edge pieces of $\Gamma$. More precisely, for an edge $e$ of $\Gamma$ with primitive direction $u_e$ and weight $w(e)$, the edge contribution to the stable intersection with $L_{\gamma,m}$ is
$
w(e)|\det(m,u_e)|.
$
The two-sided smoothing hypothesis says that the part of the contour lying over this edge has at most two real analytic sides above each weighted edge sheet. Therefore the number of transverse intersections of $L_{\gamma,m}$ with the contour over the edge $e$, away from vertex neighborhoods, is bounded by
$
2w(e)|\det(m,u_e)|.
$

Summing over all edge pieces met by $L_{\gamma,m}$ outside the vertex neighborhoods gives
$$
\#\bigl(L_{\gamma,m}\cap\CA_C\cap(\R^2\setminus U)\bigr)
\leq
2\sum_{p\in L_{\gamma,m}\cap\Gamma\cap(\R^2\setminus U)}
w(e_p)|\det(m,u_{e_p})|.
$$
The right-hand side is twice the part of the stable tropical intersection number contributed by the edge intersections outside the vertex neighborhoods.
It remains to control the intersections inside the vertex neighborhoods. Let
$
v\in\Gamma^{(0)}.
$
The affine line $L_{\gamma,m}$ may enter $U_v$ and intersect local contour branches which are not contained in the edge regions outside $U_v$. By the second part of the spine-control hypothesis, the number of such additional contour intersections inside $U_v$ is no larger than the number already forced by the adjacent weighted edges of $\Gamma$. In other words, the local vertex contribution satisfies
$$
\#(L_{\gamma,m}\cap\CA_C\cap U_v)
\leq
2\sum_{p\in L_{\gamma,m}\cap\Gamma\cap U_v}^{\mathrm{st}}
w(e_p)|\det(m,u_{e_p})|,
$$
where the superscript ``$\mathrm{st}$'' means that, if the line passes through the vertex or meets a non-transverse local configuration, the contribution is interpreted by the stable tropical intersection multiplicity obtained after a sufficiently small generic translation of the line.

Adding these estimates over all vertex neighborhoods gives
$$
\#(L_{\gamma,m}\cap\CA_C\cap U)
\leq
2\sum_{p\in L_{\gamma,m}\cap\Gamma\cap U}^{\mathrm{st}}
w(e_p)|\det(m,u_{e_p})|.
$$

Combining the estimates outside and inside the vertex neighborhoods yields
$$
\#(L_{\gamma,m}\cap\CA_C)
\leq
2\sum_{p\in L_{\gamma,m}\cap\Gamma}^{\mathrm{st}}
w(e_p)|\det(m,u_{e_p})|.
$$
By the definition of the stable tropical intersection number,
$$
\sum_{p\in L_{\gamma,m}\cap\Gamma}^{\mathrm{st}}
w(e_p)|\det(m,u_{e_p})|
=
I_m(\Gamma).
$$
Therefore
$
\#(L_{\gamma,m}\cap\CA_C)\leq 2I_m(\Gamma).
$
Since this holds for every generic transverse affine logarithmic line of normal direction $m$, taking the supremum over $\gamma$ gives
$
\mathbb R\deg_m(\CA_C)\leq 2I_m(\Gamma).
$
This proves the theorem.
\end{proof}

\begin{remark}
The factor $2$ comes from the hypothesis that the contour is a two-sided real analytic smoothing of the tropical spine away from the vertex neighborhoods. The stable tropical number $I_m(\Gamma)$ counts the weighted intersections of the affine test line with the spine. The two-sided smoothing allows at most two contour branches over each such weighted tropical intersection. The vertex hypothesis is precisely what prevents additional local contour branches near vertices from increasing the count beyond this two-sided edge contribution.
\end{remark}
  
%%%%%%%%%%%%%%%%%%%%%%%%%%%%%%%%%%%%%%%%%%%%%%%%%%%%%%%%%%%%%%%%%%%%%%%%%%%%

\section{How Patchworking Controls Contour Degrees Through Tropical Widths and Logarithmic Gauss Maps}

Let $C_t\subset(\C^\ast)^2$ be a real Viro patchworked curve with Newton polygon $\Delta$, regular subdivision $\tau$, sign distribution $s$, and tropical limit
$
\Gamma=\Trop(C_t).
$
The curve $\Gamma$ is the tropical spine dual to the subdivision $\tau$. The fundamental point is that, for sufficiently small $t>0$, the real curve $C_t$ is assembled from local algebraic pieces corresponding to the cells of $\tau$, while its amoeba is concentrated near the tropical spine $\Gamma$. Thus the large-scale geometry of the amoeba and many features of its contour are controlled by the combinatorics of $\Gamma$.

Recall that for a smooth plane curve $C_t=\{f_t=0\}$, the critical locus of $\Log|_{C_t}$ is described by the logarithmic Gauss map
$
\gamma_{\log}:C_t\longrightarrow\mathbb CP^1,\,  (z,w)\longmapsto [zf_z:wf_w].
$
As we know that 
$
\Crit(\Log|_{C_t})=\gamma_{\log}^{-1}(\mathbb RP^1), 
$
therefore the contour is the logarithmic image of the real fibers of $\gamma_{\log}$:
$
\CA_{C_t}=\Log(\gamma_{\log}^{-1}(\mathbb RP^1)).
$

Patchworking enters because the local form of $f_t$ near an edge of the tropical spine is asymptotically binomial. If an edge $e$ of $\Gamma$ has primitive direction $u_e$ and tropical weight $w_e$, then the local patchworking chart near $e$ is controlled by a binomial supported on the dual edge of the subdivision. In maximal real phase, this binomial chart produces real branches whose contour is a two-sided smoothing of the tropical edge $e$.
Recall that  maximal real phase means that the sign distribution is chosen so that each local chart is the standard real pair-of-pants model, and the real branches coming from adjacent triangles glue together without creating unnecessary cancellations. As a consequence, the real part of the algebraic curve is as large as permitted by the patchworking construction, and its amoeba contour is locally a two-sided smoothing of the corresponding tropical curve.

Now fix a primitive direction $m\in\Z^2$ and consider affine lines
$
L_{\gamma,m}=\{x\in\R^2\mid \langle m,x\rangle=\gamma\}.
$
The directional contour degree is
$
\mathbb R\deg_m(\CA_{C_t})
=
\sup_{\gamma\in\R}\#(L_{\gamma,m}\cap\CA_{C_t}),
$
with transverse intersections counted.

The tropical intersection number of $L_{\gamma,m}$ with the spine $\Gamma$ is
$
I_m(\Gamma)
=
\sum_{p\in L_{\gamma,m}\cap\Gamma}
w_{e(p)}|\det(m,u_{e(p)})|.
$
By the duality between tropical curves and Newton polygons,
$
I_m(\Gamma)=\width_{m^\perp}(\Delta),
$
where\\
$
\width_{m^\perp}(\Delta)=\max_{q\in\Delta}\langle m^\perp,q\rangle-\min_{q\in\Delta}\langle m^\perp,q\rangle.
$
Thus $\width_{m^\perp}(\Delta)$ measures the weighted number of tropical intersections of the testing line with the tropical spine (the proof can be found in Appendix D).
The reason a factor $2$ appears in contour-degree estimates is that an edge of the tropical spine usually gives a two-sided contour smoothing. Hence each stable tropical intersection contributes at most two nearby real intersections of $L_{\gamma,m}$ with the amoeba contour. This gives the expected estimate
$
\mathbb R\deg_m(\CA_{C_t})\leq 2\,I_m(\Gamma)+E_m.
$
Using $I_m(\Gamma)=\width_{m^\perp}(\Delta)$, one obtains
$$
\mathbb R\deg_m(\CA_{C_t})\leq 2\,\width_{m^\perp}(\Delta)+E_m.
$$

The correction term $E_m$ has two sources. The first source comes from vertices of the tropical spine. Near a vertex of $\Gamma$, the local equation is not binomial but a polynomial supported on the dual two-dimensional cell of $\tau$. If the subdivision is unimodular and the real phase is maximal, this local polynomial is equivalent to a real trinomial and creates no extra local contour components. In that case the vertex correction vanishes. If the local cell is non-unimodular or the real phase is not maximal, then additional local arcs or compact contour ovals may appear, and these are recorded by the vertex part of $E_m$.

The second source is the ramification of the logarithmic Gauss map. If the tested direction corresponds to a branch value of $\gamma_{\log}$, then contour branches may become tangent to the testing affine line or merge locally. This creates extra multiplicity or changes the transverse intersection count. This contribution is measured by
$$
R_m(\gamma_{\log})
=
\sum_{p\in\Ram(\gamma_{\log})\cap\gamma_{\log}^{-1}(\Lambda_m)}
\operatorname{ord}_p(\Ram(\gamma_{\log})),
$$
where $\Lambda_m\in\mathbb RP^1$ is the logarithmic tangent direction determined by the normal vector $m$.

Thus the refined patchworking estimate has the form
$
\mathbb R\deg_m(\CA_{C_t})
\leq
2\,\width_m(\Delta)+E_m(\tau,s)+R_m(\gamma_{\log}).
$
For primitive maximal patchworkings and generic directions $m$, one typically has
$
E_m(\tau,s)$ $=0
$
and
$
R_m(\gamma_{\log})=0.
$
Then the estimate becomes the clean tropical width bound
$
\mathbb R\deg_m(\CA_{C_t})\leq 2\,\width_{m^\perp}(\Delta).
$
In favorable cases this bound is sharp. For a generic line, $\Delta=\Delta_2$ and
$
\width_{(1,-1)}(\Delta_2)$ $=2.
$
Therefore
$
2\,\width_{(1,-1)}(\Delta_2)=4,
$
and indeed
$
\mathbb R\deg(\CA_C)=4.
$
For the curve
$
f(z,w)=2z^2-3z+3-w,
$
the Newton polygon is
$
\Delta=\Conv\{(0,0),(2,0),(0,1)\},
$
and
$
\width_{(1,-1)}(\Delta)=3.
$
Hence the refined tropical prediction gives
$
2\,\width_{(1,-1)}(\Delta)=6,
$
which is attained by the contour in the direction $m=(1,-1)$.
This explains the role of patchworking in contour-degree estimates. Patchworking gives a tropical spine $\Gamma$ and local algebraic charts. The tropical spine gives the leading term through the lattice width $\width_{m^\perp}(\Delta)$. The logarithmic Gauss map controls the contour itself and detects possible ramification corrections. When the patchworking is primitive, maximal, and transverse, the contour degree is governed by the simple formula
$
\mathbb R\deg_m(\CA_{C_t})\leq 2\,\width_{m^\perp}(\Delta).
$

%%%%%%%%%%%%%%%%%%%%%%%%%%%%%%%%%%%%%%%%%%%%%%%%%%%%%%%%%%%%%%%%%%%%%%%%%%%%
 %%%%%%%%%%%%%%%%%%%%%%%%%%%%%%%%%%%%%%%%%%%%%%%%%%%%%%%%%%%%%%%%%%%%%%%%%%%%

\section{Finite partition of directions and rational-direction lower bounds}

Let
$
C\subset(\mathbb C^\ast)^2
$
be a smooth real algebraic curve, and let
$
X=\CA_C\subset\R^2
$
be its amoeba contour. For
$
m\in S^1
$
and
$
\gamma\in\R,
$
write
$
L_{\gamma,m}=\{x\in\R^2:\langle m,x\rangle=\gamma\}.
$
The directional contour degree is
$$
\mathbb R\deg_m(X)=\sup_\gamma \#(L_{\gamma,m}\cap X),
$$
where the supremum is taken over generic values of $\gamma$ for which the intersection is finite and transverse. The full real contour degree is
$
\mathbb R\deg(X)=\sup_{m\in S^1}\mathbb R\deg_m(X).
$
The rational unit directions are
$
S^1_\mathbb Q=S^1\cap\mathbb Q^2.
$
Since
$
S^1_\mathbb Q\subset S^1,
$
one always has
$\di
\sup_{m\in S^1_\mathbb Q}\mathbb R\deg_m(X)\leq \mathbb R\deg(X).
$
This inequality is the general statement. Equality requires an additional stability condition.

%%%%%%%%%%%%%%%%%%%%%%%%%%%%%%%%%%%%%%%%%%%%%%%%%%%%%%%%%%%%%%%%%%%%%%%%%%%%%

\subsection*{Real Logarithmic Conormal Incidence of a Plane Curve} Let's start by giving some precise definitions.
Let
$
C=\{f=0\}\subset(\C^\ast)^2
$
be a smooth algebraic curve, where
$
f\in\C[z^{\pm1},w^{\pm1}]
$
is a Laurent polynomial. The logarithmic tangent space at a point
$
p=(z,w)\in C
$
is the image of the ordinary tangent space under the logarithmic differential
$
d\log:T_p(\C^\ast)^2\longrightarrow\C^2,
\,
(\xi_z,\xi_w)\longmapsto
\left(\frac{\xi_z}{z},\frac{\xi_w}{w}\right).
$
It is denoted by
$
T_p^{\log}C=d\log(T_pC)\subset\C^2.
$
Since $C$ is a smooth complex curve, $T_p^{\log}C$ is a complex line in $\C^2$. Its annihilator in the logarithmic cotangent space is a projective point in $\PP^1$ and is called the logarithmic conormal direction of $C$ at $p$.
For a hypersurface equation $f=0$, the logarithmic differential of $f$ is
$\di
d_{\log}f
=
z\frac{\partial f}{\partial z}\frac{dz}{z}
+
w\frac{\partial f}{\partial w}\frac{dw}{w}.
$
Thus the logarithmic conormal direction is represented by
$\di
\left[
z\frac{\partial f}{\partial z}(z,w):
w\frac{\partial f}{\partial w}(z,w)
\right]\in\PP^1.
$
This is the logarithmic Gauss map
$
\gamma_{\log}:C\longrightarrow\PP^1,
\, 
(z,w)\longmapsto
\left[
z f_z(z,w):
w f_w(z,w)
\right].
$

\begin{definition}[Logarithmic conormal incidence]
The logarithmic conormal incidence of $C$ is
$$
\mathcal I_{\log}(C)
=
\left\{
\bigl((z,w),[\lambda:\mu]\bigr)\in C\times\PP^1
\;\middle|\;
[\lambda:\mu]=[z f_z(z,w):w f_w(z,w)]
\right\}.
$$
Equivalently,
$$
\mathcal I_{\log}(C)
=
\left\{
\bigl((z,w),[\lambda:\mu]\bigr)\in C\times\PP^1
\;\middle|\;
f(z,w)=0,\quad
\lambda\,w f_w(z,w)-\mu\,z f_z(z,w)=0
\right\}.
$$
\end{definition}

The equation
$
\lambda\,w f_w-\mu\,z f_z=0
$
means that $[\lambda:\mu]$ is proportional to the logarithmic gradient
$
[z f_z:w f_w].
$
Therefore $\mathcal I_{\log}(C)$ is the graph of $\gamma_{\log}$.

\begin{definition}[Real logarithmic conormal incidence]
The real logarithmic conormal incidence of $C$ is the locus
$
\mathcal I_{\log}^{\R}(C)
=
\gamma_{\log}^{-1}(\RP^1)
\subset C.
$
Equivalently,
$$
\mathcal I_{\log}^{\R}(C)
=
\left\{
(z,w)\in C
\;\middle|\;
[z f_z(z,w):w f_w(z,w)]\in\RP^1
\right\}.
$$
\end{definition}

This condition means that the two complex numbers
$
z f_z(z,w)
\, \text{and}\, 
w f_w(z,w)
$
are real proportional. Hence, wherever $w f_w\neq0$, it is equivalent to
$
\dfrac{z f_z(z,w)}{w f_w(z,w)}\in\R.
$
A global equation avoiding division is
$
\operatorname{Im}\left(
z f_z(z,w)\,\overline{w f_w(z,w)}
\right)=0.
$
Thus
$$
\mathcal I_{\log}^{\R}(C)
=
\left\{
(z,w)\in(\C^\ast)^2
\;\middle|\;
f(z,w)=0,\quad
\operatorname{Im}\left(
z f_z(z,w)\,\overline{w f_w(z,w)}
\right)=0
\right\}.
$$
The connection with the amoeba contour is fundamental. The logarithmic map is
$
\Log:(\C^\ast)^2\longrightarrow\R^2,
\, 
(z,w)\longmapsto(\log|z|,\log|w|).
$
A point $p\in C$ is critical for $\Log|_C$ if and only if the logarithmic tangent line $T_p^{\log}C$ contains a nonzero purely imaginary vector. This is equivalent to saying that the annihilator of $T_p^{\log}C$ is represented by a real projective covector. Since that annihilator is precisely
$
[z f_z(p):w f_w(p)],
$
one obtains
$
\Crit(\Log|_C)=\mathcal I_{\log}^{\R}(C).
$
Consequently,
$$
\CA_C
=
\Log\left(\Crit(\Log|_C)\right)
=
\Log\left(\mathcal I_{\log}^{\R}(C)\right).
$$

\begin{proposition}
Let $C=\{f=0\}\subset(\C^\ast)^2$ be smooth. Then
$
\mathcal I_{\log}^{\R}(C)
=
\Crit(\Log|_C),
$
and therefore
$
\CA_C
=
\Log\left(\mathcal I_{\log}^{\R}(C)\right).
$
\end{proposition}

\begin{proof}
Let $p=(z,w)\in C$. The logarithmic tangent space
$
T_p^{\log}C\subset\C^2
$
is a complex line. The differential of $\Log$ is the real part of the logarithmic differential, namely
$
d\Log=\Rea\circ d\log.
$
Thus $p$ is critical for $\Log|_C$ exactly when the real-linear map
$
\Rea:T_p^{\log}C\longrightarrow\R^2
$
has rank less than one. Since $T_p^{\log}C$ is a complex line, this rank drop occurs precisely when $T_p^{\log}C$ contains a nonzero purely imaginary vector. Equivalently, the logarithmic conormal line annihilating $T_p^{\log}C$ is real projective.
The logarithmic conormal line is represented by
$
[z f_z(p):w f_w(p)].
$
Hence $p$ is critical for $\Log|_C$ precisely when
$
[z f_z(p):w f_w(p)]\in\RP^1.
$
This is exactly the defining condition for $\mathcal I_{\log}^{\R}(C)$. Therefore
$
\Crit(\Log|_C)=\mathcal I_{\log}^{\R}(C),
$
and applying $\Log$ gives
$
\CA_C=\Log(\mathcal I_{\log}^{\R}(C)).
$
\end{proof}

If $f$ has real coefficients, the real locus
$
C(\R)=C\cap(\R^\ast)^2
$
is not the same thing as $\mathcal I_{\log}^{\R}(C)$. The latter is usually larger. It consists of complex points of $C$ whose logarithmic conormal direction is real. Thus the word ``real'' refers to the conormal direction, not necessarily to the point $(z,w)$.
For explicit computations one writes
$
A_f=z f_z,
\,
B_f=w f_w.
$
Then
$
\mathcal I_{\log}^{\R}(C)=\{f=0,\ \Ima(A_f\overline{B_f})=0\}.
$
If
$
z=x+iy,
\,
w=u+iv,
$
then this becomes the real system
$
\Rea f(x+iy,u+iv)=0,
$\, 
$
\Ima f(x+iy,u+iv)=0,
$\, and 
$
\Ima\left(
A_f(x+iy,u+iv)\overline{B_f(x+iy,u+iv)}
\right)=0.
$
The first two equations cut out the complex curve $C$ as a real surface in $(\C^\ast)^2$, and the third equation cuts out a real curve on that surface. Its logarithmic image is the amoeba contour  (see some examples in Appendix A).

%%%%%%%%%%%%%%%%%%%%%%%%%%%%%%%%%%%%%%%%%%%%%%%%%%%%%%%%%%%%%%%%%%%%%%%%%%%%%%%

\medskip

\begin{theorem}\label{thm:finite-partition-revised}% 
Assume that the real logarithmic conormal incidence of $C$ is locally finite over the logarithmic plane and admits a compactification for which only finitely many asymptotic directions occur. Assume also that the logarithmic conormal direction map is generically finite. Then there exists a finite partition
$
S^1=U_1\cup\cdots\cup U_N
$
and integers
$
\ell_1,\ldots,\ell_N
$
such that
$
\mathbb R\deg_m(X)=\ell_j
$
for every
$
m\in U_j.
$
Consequently,
$$
\mathbb R\deg(X)=\max_{1\leq j\leq N}\ell_j.
$$
The partition may be chosen so that the open pieces are the connected components of the complement of the exceptional direction set, while the exceptional directions themselves are included as lower-dimensional strata.
\end{theorem}

\begin{proof}
The logarithmic conormal cycle gives an intrinsic incidence map
$$
\Phi:\Con_{\log}^{\R}(C)\longrightarrow \mathbb P^1(\R)\times\R,
\qquad
\Phi(p,[m])=\left([m],\langle m,\Log(p)\rangle\right).
$$
For fixed $([m],\gamma)$, the fiber of $\Phi$ parametrizes the real logarithmic conormal points whose logarithmic images lie on the affine line $L_{\gamma,m}$. Thus the finite cardinality of this fiber controls
$
\#(L_{\gamma,m}\cap X).
$
The fiber cardinality can change only when the incidence fails to be locally topologically trivial. This happens at tangent directions, asymptotic directions, or directions where the logarithmic Gauss map ramifies.

Let
$
\mathcal E_{\log}\subset S^1
$
be the union of the tangent-exceptional directions, the asymptotic-exceptional directions, and the logarithmic Gauss branch directions. By the hypotheses, this exceptional set gives a finite semialgebraic stratification of the direction circle. On each connected component of
$
S^1\setminus\mathcal E_{\log},
$
the incidence map is locally a finite topologically trivial family over the relevant open part of the parameter space. Hence the number of intersections with a generic affine line varies locally constantly as $m$ varies in that component. Therefore the directional degree
$
m\longmapsto \mathbb R\deg_m(X)
$
is constant on each such component.
The exceptional directions are added as separate strata. Since there are finitely many strata, this gives a finite partition
$
S^1=U_1\cup\cdots\cup U_N
$
and integers $\ell_j$ such that
$
\mathbb R\deg_m(X)=\ell_j
$
on $U_j$. Taking the supremum over all $m\in S^1$ gives
$
\mathbb R\deg(X)=\max_{1\leq j\leq N}\ell_j,
$
because only finitely many values occur.
\end{proof}

\begin{corollary}[Rational directions give a lower bound]\label{cor:rational-lower-bound-revised}
Under the hypotheses of Theorem~\ref{thm:finite-partition-revised}, one has
$
\sup_{m\in S^1\cap\mathbb Q^2}\mathbb R\deg_m(X)
=
\max\{\ell_j:U_j\cap(S^1\cap\mathbb Q^2)\neq\varnothing\}.
$
In particular,
$$
\sup_{m\in S^1\cap\mathbb Q^2}\mathbb R\deg_m(X)
\leq
\mathbb R\deg(X).
$$
Equality holds if at least one stratum $U_j$ with
$
\ell_j=\max_{1\leq k\leq N}\ell_k
$
meets $S^1\cap\mathbb Q^2$.
\end{corollary}

\begin{proof}
By Theorem~\ref{thm:finite-partition-revised}, the value of $\mathbb R\deg_m(X)$ is equal to $\ell_j$ on the stratum $U_j$. When the domain is restricted from all real directions to rational unit directions, only those strata which meet
$
S^1\cap\mathbb Q^2
$
can contribute. Hence
$\di
\sup_{m\in S^1\cap\mathbb Q^2}\mathbb R\deg_m(X)
=
\sup\{\ell_j:U_j\cap(S^1\cap\mathbb Q^2)\neq\varnothing\}.
$
Since the set of values is finite, this supremum is a maximum:
$\di
\sup_{m\in S^1\cap\mathbb Q^2}\mathbb R\deg_m(X)
=
\max\{\ell_j:U_j\cap(S^1\cap\mathbb Q^2)\neq\varnothing\}.
$

Since this maximum is taken over a subcollection of the indices, it is bounded above by
$$
\max_{1\leq j\leq N}\ell_j=\mathbb R\deg(X).
$$
If a stratum carrying the full maximum meets $S^1\cap\mathbb Q^2$, then the subcollection already contains an index $j$ for which $\ell_j=\mathbb R\deg(X)$, and equality follows.
\end{proof}

\begin{corollary}[Equality under an openness or stability hypothesis]\label{cor:rational-equality-stability}
Assume the hypotheses of Theorem~\ref{thm:finite-partition-revised}. Suppose, in addition, that the maximal value of
$
m\longmapsto\mathbb R\deg_m(X)
$
is attained on a nonempty open subset of $S^1$. Equivalently, suppose there is an open stratum $U_j$ such that
$
\ell_j=\mathbb R\deg(X).
$
Then
$\di
\mathbb R\deg(X)=
\sup_{m\in S^1\cap\mathbb Q^2}\mathbb R\deg_m(X).
$
\end{corollary}

\begin{proof}
Since
$
S^1\cap\mathbb Q^2
$
is dense in $S^1$, every nonempty open subset of $S^1$ contains a rational unit direction. If the maximum is attained on an open stratum $U_j$, then
$
U_j\cap(S^1\cap\mathbb Q^2)\neq\varnothing.
$
For every $m\in U_j$, one has
$
\mathbb R\deg_m(X)=\ell_j=\mathbb R\deg(X).
$
Thus the rational-direction supremum is at least $\mathbb R\deg(X)$. The reverse inequality follows from Corollary~\ref{cor:rational-lower-bound-revised}. Hence equality holds.
\end{proof}

%%%%%%%%%%%%%%%%%%%%%%%%%%%%%%%%%%%%%%%%%%%%%%%%%%%%%%%%%%%%%%%%%%%%%%%%%%%%

%%%%%%%%%%%%%%%%%%%%%%%%%%%%%%%%%%%%%%%%%%%%%%%%%%%%%%%%%%%%%%%%%%%%%%%%%%%%

%%%%%%%%%%%%%%%%%%%%%%%%%%%%%%%%%%%%%%%%%%%%%%%%%%%%%%%%%%%%%%%%%%%%%%%%%%%%

%%%%%%%%%%%%%%%%%%%%%%%%%%%%%%%%%%%%%%%%%%%%%%%%%%%%%%%%%%%%%%%%%%%%%%%%%%%%
 
%%%%%%%%%%%%%%%%%%%%%%%%%%%%%%%%%%%%%%%%%%%%%%%%%%%%%%%%%%%%%%%%%%%%%%%%%%%%

\section{Explicit Conormal Bounds and Certified Examples}
In this section we give some explicit examples and compare our bound with the universal bound of Lang, Shapiro and Shustin, and give  an explicit conormal mixed-volume computations. 
 The bound must include the logarithmic criticality condition. A character-slice mixed volume alone counts intersections of $V$ with logarithmic affine slices, but it does not count the contour. The contour is governed by the real logarithmic conormal incidence, and therefore the relevant upper bound is a conormal bound.

For a plane curve $C=\{f=0\}\subset(\C^\ast)^2$.  Let $\Delta=\Newt(f)$. If $C$ is nondegenerate with respect to $\Delta$, then
$
\deg(\overline{\gamma}_{\log})=2\Area(\Delta)
$
and
$
g(\overline C)=\#(\Int(\Delta)\cap\Z^2).
$
By Riemann--Hurwitz,
$$
\deg\Ram(\overline{\gamma}_{\log})
=
2\deg(\overline{\gamma}_{\log})+2g(\overline C)-2
=
4\Area(\Delta)+2\#(\Int(\Delta)\cap\Z^2)-2.
$$
A safe real conormal bound is
$
\ConB
=
8\Area(\Delta)+2\#(\Int(\Delta)\cap\Z^2)-2+b_\infty(C,m),
$

where $b_\infty(C,m)$ is the visible boundary contribution. Using
$
b_\infty(C,m)\le \#(\partial\Delta\cap\Z^2)
$
gives
$$
\ConB
\le
8\Area(\Delta)+2\#(\Int(\Delta)\cap\Z^2)-2+\#(\partial\Delta\cap\Z^2).
$$

%%%%%%%%%%%%%%%%%%%%%%%%%%%%%%%%%%%%%%%%%%%%%%%%%%%%%%%%%%%%%%%%%%%%%%%%

Let $m$ be the direction $m=(1,1)$. For a plane curve $C\subset(\C^\ast)^2$ defined by a polynomial of total degree $d$ and Newton polygon $\Delta$, Proposition~7 of \cite{LangShapiroShustin21} gives
$
\Rdeg(\mathcal C\mathcal A_C)
\leq
4d^3(4d-2)
+
\#(\partial\Delta\cap\Z^2)
-
\operatorname{Area}_{\Z}(\Delta),
$
where $\operatorname{Area}_{\Z}(\Delta)=2\Area(\Delta)$ is the normalized lattice area. Thus, for $n=2$, the Lang--Shapiro--Shustin value depends on the total degree as well as on the boundary lattice-point count and the normalized area of the actual Newton polygon.

For a hypersurface $H\subset(\C^\ast)^n$ of total degree $d$, Proposition~6  of \cite{LangShapiroShustin21}  gives
$$
\Rdeg(\mathcal C\mathcal A_H)
\leq
2^{\,2n+(n-1)(n-2)/2}
d^{\,n+1}
\left(4dn+2(n-1)^2-1\right)^{n-1}.
$$
For $n=3$, this specializes to
$
\LSS(3,d)
=
2^7d^4(12d+7)^2
=
128d^4(12d+7)^2.
$

We first consider the generic affine line
$
C=\{1+z+w=0\}\subset(\C^\ast)^2.
$
Its total degree is $d=1$, its Newton polygon is
$
\Delta_2=\operatorname{Conv}\{(0,0),(1,0),(0,1)\},
$
its boundary contains three lattice points, and its normalized area equals one. Proposition~7 of \cite{LangShapiroShustin21} therefore gives
$$
\LSS
=
4(1)^3(4-2)+3-1
=
8+2
=
10.
$$
The conormal value is
$
8\Area(\Delta_2)-2+\#(\partial\Delta_2\cap\Z^2)=4-2+3=5,
$
the B\'ezout value is $8$, and the exact real degree is $4$. Hence the corrected first row is
$
(4,5,8,10).
$

For the curve
$
C=\{w-z^2-1=0\},
$
the support is
$
\{(0,0),(2,0),(0,1)\}.
$
Consequently,
$$
\Newt(f)=\operatorname{Conv}\{(0,0),(2,0),(0,1)\},
$$
which is not equal to $2\Delta_2$. The total degree is $d=2$. The polygon has Euclidean area $1$, normalized area $2$, no interior lattice points, and four boundary lattice points. Therefore Proposition~7 gives
$$
\LSS
=
4(2)^3(8-2)+4-2
=
32\cdot6+2
=
194.
$$
The exact sparse conormal estimate is
$
\ConB
=
8\cdot1+2\cdot0-2+4
=
10.
$
Since our purpose is to exploit sparse Newton data, the row should use $\ConB=10$. The B\'ezout value  is
$
8d^2=8\cdot4=32.
$
The certified diagonal degree is $4$. Thus the   row is
$
(4,10,32,194).
$

The resulting table is as follows ($m=(1,1)$).
\begin{center}
\renewcommand{\arraystretch}{1.23}
\setlength{\tabcolsep}{2.2pt}
\footnotesize
\resizebox{\textwidth}{!}{%
\begin{tabular}{|p{3.0cm}|p{3.8cm}|c|c|c|c|}
\hline
\textbf{Example}
&
\textbf{Newton data and test}
&
\textbf{Certified degree in the stated test}
&
$\boldsymbol{\ConB}$
&
$\boldsymbol{\Bez}$
&
$\boldsymbol{\LSS}$
\\
\hline
Generic affine line
&
$\Delta=\Delta_2$, diagonal
&
$4$
&
$5$
&
$8$
&
$10$
\\
\hline
Plane curve $w-z^2-1=0$
&
$\Delta=\operatorname{Conv}\{(0,0),(2,0),(0,1)\}$, diagonal
&
$4$
&
$10$
&
$32$
&
$194$
\\
\hline
Plane simplex
&
$\Delta=4\Delta_2$, diagonal
&
not certified
&
$80$
&
$128$
&
$3580$
\\
\hline
Rectangle
&
$[0,3]\times[0,2]$, diagonal
&
not certified
&
$60$
&
$200$
&
$8998$
\\
\hline
Rectangle
&
$[0,10]\times[0,1]$, diagonal
&
not certified
&
$100$
&
$968$
&
$223610$
\\
\hline
Box hypersurface in $(\C^\ast)^3$
&
$[0,5]\times[0,5]\times[0,1]$
&
not certified
&
$40^{\ast}$
&
$484$
&
$36208481408$
\\
\hline
Product hypersurface in $(\C^\ast)^3$
&
$[0,10]\times2\Delta_2$
&
not certified
&
$16^{\ast}$
&
$576$
&
$60518596608$
\\
\hline
Zonotope hypersurface in $(\C^\ast)^3$
&
$Z(e_1,e_2,(1,1,1))$
&
not certified
&
$8^{\ast}$
&
$100$
&
$359120000$
\\
\hline
\end{tabular}%
}
\end{center}
 
%%%%%%%%%%%%%%%%%%%%%%%%%%%%%%%%%%%%%%%%%%%%%%%%%%%%%%%%%%%%%%%%%%%%%%%%%%%%
%%%%%%%%%%%%%%%%%%%%%%%%%%%%%%%%%%%%%%%%%%%%%%%%%%%%%%%%%%%%%%%%%%%%%%%%%%%%% 

\subsection*{Certified plane examples}

For the generic affine line
$
C=\{1+z+w=0\}\subset(\C^\ast)^2,
$
the real logarithmic conormal condition is
$
\operatorname{Im}(z\overline w)=0.
$
Since $w=-1-z$, this is equivalent to $\operatorname{Im}z=0$. Thus
$
z=t,\, w=-1-t,\, t\in\R,\, t\neq0,-1.
$
The contour is
$
\CA_C=
\{(\log|t|,\log|1+t|)\mid t\in\R,\ t\neq0,-1\}.
$
It has three real branches:
$
(\log t,\log(1+t)),\, t>0,
$
$
(\log t,\log(1-t)),\, 0<t<1,
$
and
$
(\log t,\log(t-1)),\, t>1.
$
A direct one-variable analysis shows that the global real contour degree is
$
\mathbb R\deg(\CA_C)=4.
$

For
$
C=\{w-z^2-1=0\},
$
the conormal parametrization gives the real branches
$
z=t,\, w=t^2+1,\, t>0,
$
$
z=it,\, w=1-t^2,\, 0<t<1,
$
and
$
z=it,\, w=1-t^2,\, t>1,
$
after passing to absolute values in the logarithmic image. In the diagonal direction, the intersection equations are
$
t(t^2+1)=e^\gamma,
$\,
$
t(1-t^2)=e^\gamma,
$
and
$
t(t^2-1)=e^\gamma.
$
The first and third functions are strictly increasing on their relevant domains. The middle function has derivative $1-3t^2$, hence it has exactly one critical point in $(0,1)$ and gives two roots for generic small positive $e^\gamma$. Therefore
$
\mathbb R\deg_{(1,1)}(\CA_C)=4.
$

%%%%%%%%%%%%%%%%%%%%%%%%%%%%%%%%%%%%%%%%%%%%%%%%%%%%%%%%%%%%%%%%%%%%%%%


\begin{thebibliography}{99}



\bibitem{Bergman71}
G.~M.~Bergman,
\emph{The logarithmic limit-set of an algebraic variety},
Transactions of the American Mathematical Society
\textbf{157} (1971), 459--469.



\bibitem{Bernstein75}
D.~N. Bernstein,
The number of roots of a system of equations,
\emph{Functional Analysis and Its Applications} \textbf{9} (1975), no.~3, 183--185.

\bibitem{BieriGroves84}
R.~Bieri and J.~R.~J.~Groves,
\emph{The geometry of the set of characters induced by valuations},
Journal f\"ur die Reine und Angewandte Mathematik
\textbf{347} (1984), 168--195.



\bibitem{ForsbergPassareTsikh00}
M.~Forsberg, M.~Passare, and A.~Tsikh,
Laurent determinants and arrangements of hyperplane amoebas,
\emph{Advances in Mathematics} \textbf{151} (2000), no.~1, 45--70.

\bibitem{Fulton93}
W.~Fulton,
\emph{Introduction to Toric Varieties},
Annals of Mathematics Studies, Vol.~131,
Princeton University Press, Princeton, NJ, 1993.

\bibitem{Fulton98}
W.~Fulton,
\emph{Intersection Theory},
2nd ed.,
Springer, Berlin, 1998.

\bibitem{GKZ94}
I.~M. Gelfand, M.~M. Kapranov, and A.~V. Zelevinsky,
\emph{Discriminants, Resultants, and Multidimensional Determinants},
Birkh\"auser, Boston, 1994.


\bibitem{GriffithsHarris78}
P.~Griffiths and J.~Harris,
\emph{Principles of Algebraic Geometry},
Wiley-Interscience, New York, 1978.

\bibitem{Khovanskii91}
A.~G. Khovanskii,
\emph{Fewnomials},
Translations of Mathematical Monographs, Vol.~88,
American Mathematical Society, Providence, RI, 1991.

\bibitem{LangShapiroShustin21}
L.~Lang, B.~Shapiro, and E.~Shustin,
On the number of connected components of the complement of real amoebas,
\emph{Proceedings of the London Mathematical Society} \textbf{122} (2021), no.~3, 517--544.

\bibitem{MaclaganSturmfels15}
D.~Maclagan and B.~Sturmfels,
\emph{Introduction to Tropical Geometry},
Graduate Studies in Mathematics, Vol.~161,
American Mathematical Society, Providence, RI, 2015.

\bibitem{Mikhalkin00}
G.~Mikhalkin,
Real algebraic curves, the moment map and amoebas,
\emph{Annals of Mathematics} \textbf{151} (2000), no.~1, 309--326.

\bibitem{Mikhalkin04}
G.~Mikhalkin,
Amoebas of algebraic varieties and tropical geometry,
in \emph{Different Faces of Geometry},
International Mathematical Series, Vol.~3,
Kluwer Academic/Plenum Publishers, New York, 2004, pp.~257--300.

\bibitem{NissePassare17}
M.~Nisse and M.~Passare,
Amoebas and Coamoebas of Linear Spaces,
in M.~Andersson, J.~Boman, C.~Kiselman, P.~Kurasov, and R.~Sigurdsson (eds.),
\emph{Analysis Meets Geometry},
Trends in Mathematics,
Birkh\"auser, Cham, 2017, pp.~63--80.

\bibitem{NisseSottile13}
M.~Nisse and F.~Sottile,
The phase limit set of an algebraic variety,
\emph{Algebra \& Number Theory} \textbf{7} (2013), no.~2, 339--352.

\bibitem{NisseSottile22}
M.~Nisse and F.~Sottile,
Describing amoebas,
\emph{Pacific Journal of Mathematics} \textbf{317} (2022), no.~1, 187--205.

\bibitem{PassareRullgard04}
M.~Passare and H.~Rullg{\aa}rd,
Amoebas, Monge--Amp\`ere measures, and triangulations of the Newton polytope,
\emph{Duke Mathematical Journal} \textbf{121} (2004), no.~3, 481--507.

\bibitem{Tevelev07}
J.~Tevelev,
Compactifications of subvarieties of tori,
\emph{American Journal of Mathematics} \textbf{129} (2007), no.~4, 1087--1104.

\end{thebibliography}
\end{document}